\documentclass{article}
\usepackage{
amsmath,
amsfonts,
latexsym, 
amsrefs,
amssymb
}

\usepackage{bm}

\usepackage{tocloft}

\usepackage{cite}

\usepackage{pstricks}

\usepackage{subcaption}

\usepackage{tikz}
\usepackage{tikz,fullpage}
\usetikzlibrary{shapes,arrows,%
                petri,%
                topaths, automata}%
\usepackage{tkz-berge}

\usepackage{graphicx}
\usepackage{enumerate}
\usepackage{mathrsfs}
\usepackage{wrapfig}

\usepackage{color}

\numberwithin{equation}{section}

\newcommand{\bR}{{\mathbb R}}

\newcommand{\bZ}{{\mathbb Z}}

\newcommand{\beq}{\begin{equation}}
\newcommand{\bEq}{\end{equation}}

\newcommand{\bx}{{\bf{x}}}

\newcommand{\bu}{{\bf{u}}}
\newcommand{\bv}{{\bf{v}}}
\newcommand{\bw}{{\bf{w}}}

\newcommand{\al}{\alpha}

\newcommand{\be}{\begin{equation}}
\newcommand{\ee}{\end{equation}}

\newcommand{\e}{{\varepsilon}}

\usepackage{amsmath} 
\usepackage{amssymb}
\usepackage{amsthm}

\renewcommand{\b}[1]{\bm{\mathrm{#1}}} 

\renewcommand{\cal}{\mathcal}

\newcommand{\wt}{\widetilde}

\newcommand{\ii}{\mathrm{i}} 

\newcommand{\deq}{\mathrel{\mathop:}=}

\renewcommand{\epsilon}{\varepsilon}
\renewcommand{\leq}{\leqslant}
\renewcommand{\geq}{\geqslant}

\renewcommand{\le}{\leq}
\renewcommand{\ge}{\geq}

\renewcommand{\P}{\mathbb{P}}
\newcommand{\E}{\mathbb{E}}
\newcommand{\R}{\mathbb{R}}
\newcommand{\C}{\mathbb{C}}
\newcommand{\N}{\mathbb{N}}
\newcommand{\Z}{\mathbb{Z}}

\newcommand{\abs}[1]{\lvert #1 \rvert}

\DeclareMathOperator{\var}{Var}

\DeclareMathOperator{\re}{Re}
\DeclareMathOperator{\im}{Im}

\DeclareMathOperator{\OO}{O}

\theoremstyle{plain} 
\newtheorem{theorem}{Theorem}[section]
\newtheorem*{theorem*}{Theorem}
\newtheorem{lemma}[theorem]{Lemma}

\newtheorem*{lemma*}{Lemma}

\newtheorem*{corollary*}{Corollary}

\newtheorem*{proposition*}{Proposition}

\newtheorem{definition}[theorem]{Definition}
\newtheorem*{definition*}{Definition}

\newtheorem*{example*}{Example}
\newtheorem{remark}[theorem]{Remark}

\newtheorem*{remark*}{Remark}
\newtheorem*{remarks*}{Remarks}

\makeatletter
\renewcommand{\subsection}{\@startsection
{subsection}
{2}
{0mm}
{-\baselineskip}
{0 \baselineskip}
{\normalfont\bf\itshape}} 
\makeatother

\usepackage{dsfont}
\usepackage{stmaryrd}

\newcommand{\nc}{\normalcolor}

\def\bR{{\mathbb R}}
\def\bZ{{\mathbb Z}}

\renewcommand{\b}[1]{\boldsymbol{\mathrm{#1}}} 

\def\@empty{}

\def\author#1{\par
    {\centering{\authorfont#1}\par\vspace*{0.05in}}
}

\def\titlefont{\fontsize{13}{15}\bfseries\boldmath\selectfont\centering{}}
\def\authorfont{\fontsize{13}{15}}
\def\abstractfont{\fontsize{8}{10}}

\let\affiliationfont\rhfont

\def\address#1{\par
    {\centering{\affiliationfont#1\par}}\par\vspace*{11pt}
}

\def\body{
\setcounter{footnote}{0}
\def\thefootnote{\alph{footnote}}
\def\@makefnmark{{$^{\rm \@thefnmark}$}}
}

\def\title#1{
    \thispagestyle{plain}
    \vspace*{-14pt}
    \vskip 79pt
    {\centering{\titlefont #1\par}}%
    \vskip 1em
}

\renewenvironment{abstract}{\par%
    \vspace*{6pt}\noindent 
    \abstractfont
    \noindent\leftskip18pt\rightskip18pt
}{%
  \par}

\renewcommand{\b}[1]{\boldsymbol{\mathrm{#1}}} 

\usepackage{tocloft}

\makeatletter
\renewcommand{\section}{\@startsection
{section}
{1}
{0mm}
{-2\baselineskip}
{2\baselineskip}
{\normalfont\large\scshape\centering}} 
\makeatother

\newcommand{\tnorm}[1]{{\left\vert\kern-0.25ex\left\vert\kern-0.25ex\left\vert #1 
    \right\vert\kern-0.25ex\right\vert\kern-0.25ex\right\vert}}

\begin{document}

\title{Random band matrices in the delocalized phase, II:\\  Generalized resolvent estimates}

{\let\thefootnote\relax\footnotetext{\noindent The work of P.B. is partially supported by the NSF grant DMS\#1513587. 
The work of H.-T. Y. is partially supported by NSF Grant  DMS-1606305 and a Simons Investigator award.
The work of  J.Y. is partially supported by the NSF grant DMS\#1552192.}}

\vspace{1.2cm}
~\hspace{-1cm}
\noindent\begin{minipage}[b]{0.25\textwidth}

 \author{P. Bourgade}

\address{Courant Institute\\
   bourgade@cims.nyu.edu}
 \end{minipage}
 \begin{minipage}[b]{0.25\textwidth}

\author{F. Yang }
\address{U. of California, Los Angeles\\
   fyang75@math.ucla.edu}
 \end{minipage}
\begin{minipage}[b]{0.25\textwidth}

 \author{H.-T. Yau}

\address{Harvard University\\
   htyau@math.harvard.edu}

 \end{minipage}
\begin{minipage}[b]{0.25\textwidth}

 \author{J. Yin}

\address{U. of California, Los Angeles\\
    jyin@math.ucla.edu}

 \end{minipage}

\begin{abstract}
This is the second part of a three part series abut delocalization for band matrices. In this paper, we consider a general class of $N\times N$ random band matrices $H=(H_{ij})$ whose entries are centered random variables, independent up to a symmetry constraint. We assume that the variances $\mathbb E |H_{ij}|^2$ form a band matrix with typical band width $1\ll W\ll N$. We consider the generalized resolvent of $H$ defined as $G(Z):=(H - Z)^{-1}$, where $Z$ is a deterministic diagonal matrix such that $Z_{ij}=\left(z\mathds{1}_{1\le i \le W}+\wt z\mathds{1}_{ i > W} \right) \delta_{ij}$, with two distinct spectral parameters $z\in \mathbb C_+:=\{z\in \mathbb C:\im z>0\}$ and $\wt z\in \mathbb C_+\cup \mathbb R$. In this paper, we prove a sharp bound for the local law of the generalized resolvent $G$ for $W\gg N^{3/4}$. This bound is a key input for the proof of delocalization and
bulk universality of random band matrices in \cite{PartI}. Our proof depends on a fluctuations averaging bound on certain averages of polynomials in the resolvent entries, which will be proved in \cite{PartIII}.
\end{abstract}

\tableofcontents

\section{The model and the results.}

\subsection{The model.}\ Our {\rm goal} in this  paper is to establish estimates on Green's functions which  were used in the proof of delocalization conjecture and
bulk universality for random band matrices.  All results in this paper apply to both real and complex band matrices. For simplicity of notations, we consider only the real symmetric case. 
Random band matrices are  characterized by the property that the matrix element $H_{ij}$   becomes negligible if $\mbox{dist}(i,j)$ exceeds the band width $W$.  We shall restrict ourselves to the convention that $i,j \in \Z_N=\Z \cap (-N/2, N/2],$ and $i-j$ is defined modular $N$. More precisely, we consider the following matrix ensembles.

\begin{definition}[Band matrix $H_N$ with bandwidth $W_N$]\label{jyyuan}
Let $H_N$ be an $N\times N$ matrix with real centered entries ($H_{ij}$: $i,j\in \Z_N)$ which are  
independent up to the condition $H_{ij}= H_{ji}$.  We say that $H_N$ is a random band matrix with (typical) bandwidth $W=W_N$ if
\be\label{bandcw0}
s_{ij} :=\E |H _{ij}|^2  =f(i-j)
\ee
for some non-negative symmetric function $f: \Z_N\to \R_+$ satisfying 
\begin{equation}\label{sumsone}
\sum_{x\in\mathbb{Z}_N} f(x)=1, 
\end{equation}
and there exist some (small) positive constant $c_s$ and (large) positive constant $C_s$ such that 
\begin{equation}\label{bandcw1}
c_s\, W^{-1} \cdot \mathds{1} _{|x| \le   W}\le  f(x)  \le C_s \,W^{-1} \cdot \mathds{1} _{|x| \le C_s W}, \quad i,j\in \Z_N.
\end{equation}
\end{definition}   

The method in this paper also allows to treat cases with 
exponentially small mass away from the band width (e.g. $f(x)\leq C_s W^{-1}e^{-c_s{|x|^2}/{W^2}}$). We work under the hypothesis (\ref{bandcw1}) mainly for simplicity.

\vspace{5pt}

We assume that the random variables $H_{ij}$ have arbitrarily high moments, in the sense that for any fixed $p\in \mathbb N$, there is a constant $\mu_p>0$ such that
\begin{equation}\label{eqn:subgaus}
\max_{i,j}\left(\mathbb E|H_{ij}|^p\right)^{1/p} \le \mu_p \var \left(H_{ij}\right)^{1/2}
\end{equation}
uniformly in $N$. 

In this paper, we will {\it not} need the following moment condition assumed in Part I of this series \cite{PartI}:  there is fixed $\e_m>0$ such that for $|i-j|\le W$,  
$
\min_{  |i-j|\le W}\left(\E \,\xi^4_{ij}-(\E \,\xi^3_{ij})^2-1 \right)\ge N^{-\e_m},
$
where $ \xi _{ij}:= H_{ij} (s_{ij})^{-1/2}$  is the normalized random variable with mean zero and variance one.

All the results in this paper will  depend  on the parameters $c_s$, $C_s$ in \eqref{bandcw1}  and $\mu_p$ in \eqref{eqn:subgaus}.  But we will not track the dependence on $c_s$, $C_s$  and $\mu_p$ in the proof.

Denote the eigenvalues of $H_N$ by
$
\lambda_1\leq\dots\leq \lambda_N.
$
It is well-known that the empirical spectral measure $\frac{1}{N}\sum_{k=1}^N\delta_{\lambda_k}$ converges
almost surely to the Wigner semicircle law with density
$$
\rho_{\rm sc}(x)=\frac{1}{2\pi}\sqrt{(4-x^2)_+}.
$$
The aim of this paper is to estimate  ``the generalized  resolvent"   $G(  z,\wt z\, )$ of $H_N$  defined by 
\be\label{defGzetag0}
  G(z,\wt z\, ):=  
\left(H_N- \begin{pmatrix}  zI_{W\times W}  & 0 \cr 0 & \wt z I_{(N- W)\times (N- W)}  \end{pmatrix} \right)^{-1}, \quad  z,\; \wt z\in \C^+\cup \R,
 \ee
where $\C^+$ denotes the upper half complex plane $\C^+:=\{z\in\C:\im z>0\}$. The generalized  resolvent is an important quantity used in Part I of this series \cite{PartI}. The key point of this generalization, compared with the usual resolvent, is the freedom to choose different $z$ and $\wt z$. To the best of our knowledge, the local law for this type of generalized resolvent has only been studied in the preceding paper \cite{BouErdYauYin2017}, where it was assumed that $W\ge cN$ for some constant $c >0$.

To understand the role of the  generalized  resolvent,
we block-decompose  the band matrix   $H_N$ and its eigenvectors as
$$
   H_N=  \begin{pmatrix} A  & B^* \cr B & D  \end{pmatrix}, \quad 
   \b\psi_j:=   \begin{pmatrix}\b w_j \cr \b p_j \end{pmatrix},
$$
where $A$ is a $W\times   W$ Wigner matrix. From the eigenvector equation $H \b\psi_j = \lambda_j \b\psi_j$, we get
$$
  Q_{\lambda_j}  \bw_j  = \lambda_j  \bw_j, \quad   Q_e :=  A
 - B^* \frac{1}{D-e}B.
$$
Thus $\bw_j$  is an eigenvector of  $Q_e:= A- B^* (D  - e)^{-1}B$  with eigenvalue $\lambda_j$ when   $e= \lambda_j$. 
A key input to the proof of universality and QUE for random band matrices is an estimate on the Green's function of $Q_e$. 
Since some eigenvalues of $D$  can be very close to $e$, the  matrix  $(D  - e)^{-1}$ can be very singular.  It is thus very difficult (if possible) to estimate  the Green's function of $Q_e$ directly. On the other hand, the Green's function of $Q_e$ is just the $W\times W$ minor of the generalized  resolvent  $G(  z, e )$ of $H_N$, which we find to be relatively more doable. 

Due to the need in Part I, we will consider generalized resolvent for a general class of band matrices.
More precisely, we introduce the following Definition \ref{defHg}. Here and throughout the rest of this paper, we will use the notation that for any $a,b\in \Z$, 
$$\llbracket a,b \rrbracket:=[a,b]\cap \Z.$$

 \begin{definition}[Definition of $H_\zeta^{\b g}$]\label{defHg}
For any sufficiently small $\zeta>0$ and any $\b g=(  g_1,   g_2, \cdots,  g_N)\in \R^N$, $H_{\zeta}$ and $H_\zeta^{\b g}$ will denote $N \times N$ real symmetric matrices satisfying the following properties. The entries $(H_\zeta)_{ij}$ are centered and independent up to the symmetry condition, satisfy \eqref{eqn:subgaus}, and have variances
$$
\mathbb E  |(H_{\zeta})_{ij}|^2 =(s_{\zeta})_{ ij}:=s_{ij}-\frac{\zeta (1+\delta_{ij})}{W}{\bf 1}_{ i, j\in\llbracket 1, W\rrbracket },
$$
where $s_{ij}$, $i,j\in \Z_N$, satisfy the conditions in Definition \ref{jyyuan}. 
Then the matrix $H_\zeta^{\b g}$ is defined by
$$
(H_\zeta^{\b g} )_{ij}:= (H_\zeta)_{ij} -  g_i \delta_{ij}. 
$$
 We denote by $S_0$ and  $\Sigma $ the matrices with entries $(S_0)_{ij} = s_{ij}$ and  $\Sigma_{ij} = \frac{ (1+\delta_{ij})}{W}{\bf 1}_{ i, j\in\llbracket 1, W\rrbracket }$, respectively. Then the matrix of variances is
$$S_\zeta := S_0 -\zeta\Sigma, \quad (S_\zeta)_{ij}=(s_\zeta)_{ij}.$$ 
 \end{definition}

\subsection{The results.}\ The generalized resolvent $G_\zeta^{\b g}(  z,\wt z\, )$ of $H_\zeta^{\b g}$ is defined similarly as in \eqref{defGzetag0} by  
$$
  G_\zeta^{\b g}(z,\wt z\, ):=  
\left(H_\zeta^{\b g}- \begin{pmatrix}  zI_{W\times W}  & 0 \cr 0 & \wt z I_{(N- W)\times (N- W)}  \end{pmatrix} \right)^{-1}.
$$
Define $\big((M_\zeta^{\b g})_i(  z, \wt z)\big)_{i=1}^N$ as the solution vector to the system of self-consistent equations 
 \be\label{falvww}
\left((M_\zeta^{\b g})_i(  z, \wt z)\right)^{-1}=-z\mathds{1}_{i\in\llbracket 1,W\rrbracket}- \wt z\mathds{1}_{i\notin\llbracket 1,W\rrbracket}-{g}_i- \sum_{j} (s_{\zeta})_{ ij}(M_\zeta ^{\b g})_j(z,\wt z ), 
 \ee
for $z,\wt z\in \C^+\cup \R$ and $i\in \Z_N$, with the constraint that $$(M_0^{\b 0})_i( \wt z, \wt z\,)=m_{\rm sc} (\wt z+ {\rm i} 0^+ ),$$
 where $m_{\rm sc}$ denotes the Stieltjes transform of the semicircle law 
\be\label{msc}
m_{\rm sc}(z):=\frac{-z+\sqrt{z^2-4}}{2},\quad z\in \C^+.
\ee
(The existence, uniqueness and continuity of the solution is given by Lemma \ref{UE} below.) For simplicity of notations, we denote by  $M_\zeta^{\b g}( z, \wt z)$ the diagonal matrix with entries 
$$(M_\zeta^{\b g})_{ij}: =(M_\zeta^{\b g})_i\delta_{ij}.$$ We will show that $M_\zeta^{\b g}( z, \wt z)$ is the asymptotic limit of the generalized resolvent 
$G_\zeta^{\b g}(  z,\wt z\, )$. 
We now list some properties of  $M_\zeta^{\b g}$  needed for the proof of local law stated in Theorem \ref{LLniu}.  Its proof is delayed to Section \ref{matrix-norm}.

\begin{lemma}\label{UE}
 Assume $|\re \wt z\, |\le 2-\kappa$  and $| \wt z| \le \kappa^{-1}$ for some (small) constant $\kappa>0$. Then there exist constants $ c, C>0$ such that the following statements hold.
 \begin{itemize}
\item (Existence and Lipschitz continuity) If  
\be\label{heiz}
 \zeta+\| \b g\|_\infty+ |z-\wt z|\le c, 
\ee
 then there exist $(M_\zeta^{\b g})_i(z, \wt z)$, $ i \in \bZ_N$, which satisfy \eqref{falvww} and 
\be\label{bony}
 \max_i\left|(M_\zeta^{\b g})_i(z, \wt z)-m_{\rm sc}(\wt z+  {\rm i}0^+ )\right|\le C \left(\zeta+\| \b g\|_\infty+ |z-\wt z|\,\right).
\ee
If, in addition, we have $\zeta'+\| \b g'\|_\infty+ |z'-\wt z\,'| \le c $, then 
\be\label{dozy}
\max_i\left|(M_{\zeta'}^{ \b g'})_i(z', \wt z\,')-(M_\zeta^{\b g})_i(z, \wt z)\right|\le C\left(\| \b g-\b g'\|_\infty +|z'-z|+|\wt z\,'-\wt z|+|\zeta'-\zeta|\right).
\ee

\item (Uniqueness) The solution vector $\big((M_\zeta^{\b g}\big)_i(z, \wt z)\big)_{i=1}^N$ to \eqref{falvww}  is unique under \eqref{heiz} and the constraint 
$$
\max_i\left|(M_\zeta^{\b g})_i(z, \wt z)-m_{\rm sc}(\wt z+ {\rm i}0^+  )\right|\le c.
$$
\end{itemize} 
\end{lemma}

We now state our results on the generalized resolvent of $H_\zeta^{\b g}$. In this paper, we will always use $\tau$ to denote an arbitrarily small positive constant independent of $N$, and $D$ to denote an arbitrarily large positive constant independent of $N$. Define for any matrix $X$ the max norm 
$$\| X \|_{\max} := \max_{i, j} |X_{ij}|.$$ 
The notations $\eta_*, \eta^*$ and $r$ in next theorem were used in Assumptions 2.3 and 2.4 
of Part I of this series \cite{PartI}. Their meanings are not important for this paper and the reader can simply view them as some parameters.  In this paper, all the statements hold for sufficiently large $N$ and we will not repeat it everywhere.

 \begin{theorem}[Local law]\label{LLniu}
 Define a set of parameters with some constants $\epsilon_*,\epsilon^*>0$:
\be\label{mouyzz} 
\eta_*:=N^{-\e_*},\quad \eta^*:=N^{-\e^*} \quad r:=N^{- \e_*+3\e^*},\quad  T:=N^{- \e_*+\e^* }, \quad  0<\e^*\le \e_* /20.
\ee
Fix any $|e|<2-\kappa$ for some constant $\kappa>0$. Then for any deterministic $z$, $\zeta$, $\b g$ satisfying 
  \be\label{mouyzz2} 
   |\re z-e|\le r, \quad \eta_* \le  \im z\le \eta^* , \quad 0\le \zeta\le T, \quad \|\b g\|_\infty\le W^{-3/4}  , 
\ee
and $W, \, \e_*,\, \e^*$ satisfying
\be\label{67}
\log _N W \ge \max\left\{ \frac67 + \epsilon^* , \ \frac34 + \frac34\e_*+ \e^* \right\},
\ee
we have that for any fixed $\tau>0$ and $D>0$, 
\be\label{jxw}
\P\left( \|G_\zeta^{\b g}(z, e)-M_\zeta^{\b g}( z, e)\| _{\max}\ge  N^{ \tau}\left(\frac{1}{\sqrt{W\im z}}+ \frac{N^{1/2}}{W}\right)\right)\le N^{-D}. 
 \ee 
In fact, the last estimate holds under the weaker assumption 
\be\label{jxw0}
\log _N W \ge \max\left\{ \frac 34 + \epsilon^*, \ \frac12 + \e_*+\e^* \right\}.
\ee
\end{theorem}
 
We will refer to the first statement, i.e.,  \eqref{jxw} under the assumption \eqref{67}, as the weak form of this theorem, and the statement \eqref{jxw} under assumption \eqref{jxw0} as the strong form.  This paper gives a full and self-contained proof for the weak form, which helps the reader understand the basic strategy of our proof. On the other hand, the proof for the  strong form is much more involved, and we include a substantial part into a separate paper \cite{PartIII}.
Only the strong form of Theorem \ref{LLniu} was used in part I of this series \cite{PartI}, where we took $\log _N W>3/4$, $\e_* <1/4$ and $\e^*$ to be a sufficiently small constant.

 The main purpose of this part and part III \cite{PartIII} of this series is to prove the above Theorem \ref{LLniu}. In fact, the bound (\ref{jxw}) is almost optimal under our setting  in the sense that it (at least) gives the correct size of $\mathbb E|(G_\zeta^{\b g})_{ij}|^2$ for $i\ne j$ up to an $N^\tau$ factor.
This sharp bound is very important for the proof of the complete delocalization of eigenvectors and the bulk universality of random band matrices in part I \cite{PartI}. As explained there, the bound must be of order ${\rm o}(W/N)$ to allow the application of the so-called {\it mean field reduction} method, which was introduced in \cite{BouErdYauYin2017} and is the starting point of this series. Compared with the local law for regular resolvents, the main difficulty in proving the local law for the generalized resolvents is due to the small and even vanishing imaginary part of $\wt z$. As a result, some key inputs, such as Ward's identity (see \eqref{ward}) for the regular resolvents estimates are missing. In fact, as discussed before, the case $\|G(z,\wt z)\|_{\max}=\infty$ could occur when $\wt z=e$ is real.  This difficulty has already appeared in the case $W\ge c N$ in \cite{BouErdYauYin2017}, where some "uncertainty principle" was introduced to solve this problem. Unfortunately, this method seems difficult to apply in the $W\ll N$ case. Instead, in this paper, we shall use a totally different 
strategy, i.e, the $T$-equation method, which was introduced in \cite{ErdKnoYauYin2013-band}. 
Moreover, we have to improve the induction (bootstrap) argument used in \cite{ErdKnoYauYin2013-band}, as explained below. We remark that the proofs of the weak form and strong form of Theorem \ref{LLniu} are completely parallel, except that we will apply a stronger $T$-equation estimate (Lemma \ref{YEniu}) than the one (Lemma \ref{YEniu1}) used in the proof of the weak form. We shall give a simple proof of the weak $T$-equation estimate using the standard fluctuation averaging mechanism as in the previous proof of local semicircle law \cite{ErdKnoYauYin2013,ErdYauYin2012Univ}. The proof of the strong $T$-equation estimate is based on an improved (and substantially more involved) fluctuation averaging result, whose proof is delayed to part III of this series \cite{PartIII}.
\nc

\subsection{Sketch of proof.}\  In the following discussion, for two random variables $X$ and $Y$, we shall use the notation $X\prec Y$ if for any fixed $\tau>0$, $|X|\le N^\tau|Y|$ with high probability for large enough $N$.

We define the $T$ matrix with entries
\be \label{def: T}
T_{ij}:= \sum_{k } S_{ik} |G_{kj}|^2 ,\quad G\equiv G_\zeta^{\b g}, \quad  S_{ik}\equiv (S_{\zeta})_{ik},
\ee
With a standard self-consistent equation estimate (see Lemma \ref{bneng}), one can show that
\be\label{selfT}
\left\|G - M\right\|_{\max}^2\prec \|T\|_{\max}, \quad M\equiv M_\zeta^{\b g}.
\ee
Our proof of Theorem \ref{LLniu} is based on an induction argument combined with a self-consistent $T$-equation estimate as explained below. We introduce the following notation:
\begin{align}\label{tri} 
 \tnorm{G}^2(z,\wt z) :=   \max_j   \sum_{1\le i\le N} |G_{ij}(z,\wt z)|^2 ,\quad \Lambda(z,\tilde z):=\left\|G - M\right\|_{\max}(z,\wt z).
\end{align}
Fix $z$ and $\re \wt z=e$. We perform the induction with respect to the imaginary part of $\wt z$. Define a sequence of $\wt z_n$ such that 
$$ \im \wt z_n = N^{-n\e }\im   z, \quad \re \wt z_n=e,$$ 
for small enough constant $\e>0$. In the $n=0$ case with $\im \wt z_0 = \im  z$, using the methods in \cite{ErdKnoYauYin2013,ErdYauYin2012Univ}, we can obtain the local law \eqref{jxw} for $G(z,\wt z_0)$. Suppose we has proved the local law for $G(z,\wt z_{n-1})$: 
\be\label{jxw_prec}\Lambda(z,\wt z_{n-1}) \prec \Phi_{{\rm goal}}, \quad \Phi_{{\rm goal}}:= \frac{1}{\sqrt{W\im z}}+ \frac{N^{1/2}}{W}.\ee
Then with  $\im \wt z_{n} = N^{-\e }\im   z_{n-1}$ and a simple (but quite sharp up to an $N^{2\epsilon}$ factor) $L^2$-estimate, we get a bound on the $n$-th level: 
\be\label{initial}
\tnorm{G}^2(z,\wt z_{n})\prec N\wt \Phi^2 , \quad \wt \Phi^2 :=N^{2\epsilon}\Phi^2_{{\rm goal}},
\ee
which gives a rough bound $\Phi^{(0)}$ by the self-consistent equation estimate \eqref{selfT}:
\be\label{initial2}
\quad \|T\|_{\max} (z,\wt z_{n})\le \frac{C_s}{W}\tnorm{G}^2(z,\wt z_{n}) \prec (\Phi^{(0)})^2 \Rightarrow \Lambda(z,\wt z_n) \prec  \Phi^{(0)}, \quad \Phi^{(0)}:= \sqrt{\frac{N}{W}}\wt \Phi,
\ee
where $C_s$ is the constant from \eqref{bandcw1}. Note that $\wt\Phi$ is very close to $\Phi_{{\rm goal}}$, while $\Phi^{(0)}$ is not. Now with the {strong} $T$-equation estimate (see Lemma \ref{YEniu}), one can get an improved bound $(\Phi^{(1)})^2$ on $T$ as follows: 
\be\label{Testimate}
\|T\|_{\max} (z,\wt z_{n})\prec (\Phi^{(1)})^2 \Rightarrow \Lambda(z,\wt z_n) \prec  \Phi^{(1)}, \quad \Phi^{(1)}:= \Phi^2_{{\rm goal}} + \left( \frac{N}{W\im z}+\frac{N^{2}}{W^2} \right) \left( \wt \Phi^2+ {N^{-1/2}}\right) (\Phi^{(0)})^2, 
\ee
where we used \eqref{selfT} to get a better bound $\Lambda(z,\wt z_n)\prec \Phi^{(1)}$. With \eqref{jxw0}, one can verify that $\Phi^{(1)}\le \Phi_{{\rm goal}}+N^{-\epsilon'}\Phi^{(0)}$ for some constant $\epsilon'>0$. 
After at most $l:=1/\e'$ many iterations with \eqref{Testimate} and \eqref{selfT}, i.e. $\Phi^{(0)}\to \Phi^{(1)}\to \cdots \to \Phi^{(l)}$, we can obtain the local law \eqref{jxw_prec} for $G(z,\wt z_n)$, which is used as the input for the next induction. The key point of this induction argument is that one has a good $L^2$-bound \eqref{initial} inherited from the local law on the upper level, and this $L^2$-bound can be used in the $T$-equation estimate \eqref{Testimate} to give an improved bound for $\Lambda(z,\wt z_n)$ on this level.   
Finally, after finitely many inductions in $n$, we can obtain the local law \eqref{jxw} for, say, $G(z, e+\ii N^{-10})$. Then with a continuity argument, we can prove the local law \eqref{jxw} for $G(z,e)$. In Fig.\,\ref{proof_strat}, we illustrate the flow of the induction argument with a diagram. 

We remark that the above induction argument is not a continuity argument, as used e.g. in the works \cite{ESY_local,ErdKnoYauYin2013,ErdYauYin2012Univ} on local semicircle law of regular resolvents. The multiplicative steps $\im \wt z_n \to N^{-\e }\im  \wt z_n$ that we made are far too large for a continuity argument to work. The main reason for choosing this multiplicative step is that the $T$-equation estimate can only be applied for ${\rm O}(1)$ number of times due to the degrade of the probability set (see Remark \ref{remarkY}).

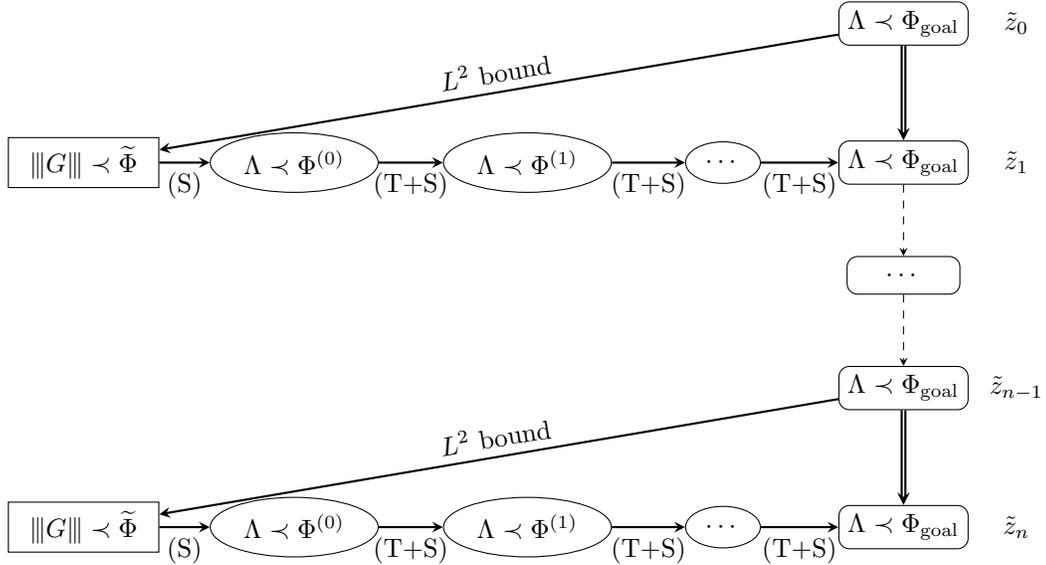
\begin{figure}[htb]
\centering
\begin{tikzpicture}[node distance=1.5cm]

\tikzstyle{startstop} = [rectangle,rounded corners, minimum width=1cm,minimum height=0.5cm,text centered, draw=black]
\tikzstyle{startstop2} = [rectangle,rounded corners, minimum width=1.5cm,minimum height=0.5cm,text centered, draw=black]
\tikzstyle{rough} = [rectangle,minimum height=0.5cm,text centered,minimum width =2cm,draw=black]
\tikzstyle{Gestimate} = [ellipse,minimum width=0.5cm,minimum height=0.5cm,text centered,draw=black]
\tikzstyle{Gestimate2} = [ellipse,minimum width=0.5cm,minimum height=0.5cm,text centered,draw=black]
\tikzstyle{arrow} = [thick,->,>=stealth]
\tikzstyle{arrow_dash} = [dashed,->,>=stealth]
\tikzstyle{arrow_double} = [double,thick, ->,>=stealth]

\node (start) [startstop] {$\Lambda \prec \Phi_{{\rm goal}}$}; \node [right of = start] {$\tilde z_0$};
\node (input1) [rough,below of=start, xshift=-10.9cm, yshift=-0.35cm] {$\tnorm{G}\prec \wt \Phi$ };
\node (process1) [Gestimate, right of=input1, xshift=1.3cm] {$ \Lambda\prec \Phi^{(0)}$};
\node (process12) [Gestimate, right of=process1, xshift=1.6cm] {$ \Lambda\prec \Phi^{(1)}$};
\node (process13) [Gestimate2, right of=process12, xshift=1.1cm] {$\cdots$};
\node (process14) [startstop, below of=start, yshift=-0.35cm] {$\Lambda \prec \Phi_{{\rm goal}}$};\node [right of = process14] {$\tilde z_1$};

\node (inputm) [startstop2,below of=process14] {$\cdots$ };

\node (processn) [startstop, below of=inputm] {$\Lambda\prec \Phi_{{\rm goal}}$};\node [right of = processn] {$\tilde z_{n-1}$};
\node (input2) [rough,below of=processn, xshift=-10.9cm, yshift=-0.35cm] {$\tnorm{G} \prec \wt \Phi$ };
\node (process2) [Gestimate, right of=input2, xshift=1.3cm] {$ \Lambda\prec\Phi^{(0)}$};
\node (process22) [Gestimate, right of=process2, xshift=1.6cm] {$ \Lambda\prec\Phi^{(1)}$};
\node (process23) [Gestimate2, right of=process22, xshift=1.1cm] {$\cdots$};
\node (process24) [startstop, below of=processn, yshift=-0.35cm] {$\Lambda \prec \Phi_{{\rm goal}}$};\node [right of = process24] {$\tilde z_n$};

\draw [arrow] (start) -- node[above,sloped]{$L^2$ bound}(input1);
\draw [arrow] (input1) --node[below]{(S)} (process1);
\draw [arrow] (process1) -- node[below]{(T+S)} (process12);
\draw [arrow] (process12) -- node[below]{(T+S)} (process13);
\draw [arrow] (process13) -- node[below]{(T+S)} (process14);
\draw [arrow_double] (start) -- (process14);
\draw [arrow_dash] (process14) -- (inputm);

\draw [arrow] (processn) -- node[above,sloped]{$L^2$ bound} (input2);
\draw [arrow] (input2) -- node[below]{(S)}(process2);
\draw [arrow] (process2) -- node[below]{(T+S)} (process22);
\draw [arrow] (process22) -- node[below]{(T+S)}(process23);
\draw [arrow] (process23) -- node[below]{(T+S)}(process24);
\draw [arrow_dash] (inputm) -- (processn);
\draw [arrow_double] (processn) -- (process24);
\end{tikzpicture}
\caption{The diagram for the induction argument with respect to $n$. At each level $n-1$, we obtain the local law \eqref{jxw_prec}, which gives the rough bound $\Phi^{(0)}$ on level $n$ through \eqref{initial} and \eqref{initial2}. Applying \eqref{Testimate} and \eqref{selfT} iteratively, one can improve the initial bound $\Phi^{(0)}$ to the sharp bound $\Phi_{{\rm goal}}$. In the diagram, (S) stands for an application of the self-consistent equation estimate \eqref{selfT}, and (T+S) stands for an application of the $T$-equation estimate \eqref{Testimate} followed by a self-consistent equation estimate \eqref{selfT}.}
\label{proof_strat}
\end{figure}

The main difficulty of our proof lies in establishing the $T$-equation estimate \eqref{Testimate}. The starting point is a self-consistent equation for the $T$ matrix, i.e. the $T$-equation, see \eqref{Tequation} below.
In this paper, we focus on proving the stability of the $T$-equation, i.e. bounding $\|\left(1-S|M|^2\right)^{-1}S\|_{\max}$ in \eqref{Tequation}, where we abbreviate $S\equiv S_{\zeta}$. For regular resolvent of generalized Wigner matrices (i.e. $\tilde z=z$, $\zeta=0$ and $\mathbf g=\bf 0$), we have $|M| \le 1- c\im z$ for some constant $c>0$. However, in our general setting and in particular when $\im \tilde z$ is small, we actually have $\|M\|_\infty>1$ and $\|S|M|^2\|_{l^\infty\to l^\infty}>1$. Therefore, the usual Taylor expansion approach cannot be used (in fact, it is not even easy to see that $1$ is outside the spectrum of $|M|^2 S$).  In this paper, we will establish the following bound
$$\left\|(1-S|M|^2)^{-1} S\right\|_{\max}=\OO \left( \frac{1}{W\im z}+\frac{N}{W^2}\right).$$
{ One important component for the proof is the estimate $\sum_{i}(|M_i|^2-1)\le -cW\im z$ for some constant $c>0$. To see this bound is useful, we can intuitively view $(|M|^2 S)^n$ as an $n$-step inhomogeneous random walk on $\Z_N$ with annihilation, where the average annihilation rate is $-W\im z/N$ by the above bound. This shows that we can explore some decay properties of $(|M|^2 S)^n$ as $n$ increase, which may give some useful bounds on the Taylor expansion of $(1-S|M|^2)^{-1}$. However, our proof actually will not follow this heuristic argument, see Section \ref{matrix-norm}.} 

Finally, to finish the proof of the  strong version of the $T$-equation estimate (Lemma \ref{YEniu}), we need a fluctuation averaging results for a quantity of the form $N^{-1}\sum_k \cal E_k$, where $\mathcal E_k$'s are some polynomials of the generalized resolvent entries.  
The proof involves a new graphical method and we include it in part III of this series \cite{PartIII}.

 \section{Tools for the proof of Theorem \ref{LLniu} }

The basic strategy to  prove Theorem \ref{LLniu} is to apply the self-consistent equation estimate: Lemma \ref{bneng}, and the $T$-equation estimate: Lemma \ref{YEniu1} or \ref{YEniu}, in turns. We collect these results in this section, and use them to prove Theorem \ref{LLniu} in next section.

For simplicity, we will often drop the superscripts $\zeta$ and $\b g$ from our notations. In particular, $G$ and $M$ are always understood as $G_\zeta^{\bf g}$ and $M_\zeta^{\bf g}$, while $H$ and $S$ are understood as $H_{\zeta}$ and $S_{\zeta}$ in the rest of this paper.

In the proof, for quantities $A_N$ and $B_N$, we will use the notations $A_N = {\rm O}(B_N)$ and $A_N \asymp B_N$ to mean that $|A_N| \le C|B_N|$ and $C^{-1}|B_N| \le |A_N| \le C|B_N|$, respectively, for some constant $C>0$.

\subsection{The self-consistent equation estimate.} \ 
The self-consistent equation estimate is the starting point of almost every proof of the local law of the (generalized) resolvents of random matrices. We now state the self-consistent equation estimate for our model.

\begin{lemma}[Self-consistent equation estimate]\label{bneng} Suppose that 
$|\re \wt z\, |\le 2-\kappa$ for some constant $\kappa>0$. Then there exists constant $c_0>0$ such that  if
$$
 \zeta+\| \b g\|_\infty+ |z-\wt z|\le c_0,
$$
then the following statement holds. If there exist some fixed $\delta>0$ and some deterministic parameter $\Phi \ge W^{-1/2}$ such that 
\be\label{cllo}
  \|G(z,\wt z)-M(z,\wt z)\|_{\max}\le N^{-\delta}, \quad  \|T\|_{\max}\le \Phi^2, 
\ee 
in a subset $\Omega$ of the sample space of the random matrices, then for any fixed $\tau>0$ and $D>0$, 
\be\label{34}
\P\left({\bf 1}_{  \Omega} \|G(z,\wt z)-M(z,\wt z)\|_{\max}\ge N^\tau \Phi \right)\le N^{-D}.
\ee 
\end{lemma}
Note that by the definition of $T$-matrix in \eqref{def: T}, we have
$$ \|T\|_{\max} \le \|G(z,\wt z)-M(z,\wt z)\|_{\max}^2 + {\rm O}(W^{-1}).$$
Hence we can always choose $\Phi =\OO (N^{-\delta})$ in \eqref{cllo}. The proof of Lemma \ref{bneng} follows the standard idea of using a vector-level self consistent equation method \cite{ErdYauYin2012Univ,ErdKnoYauYin2013}. In preparation for the proof, we recall the following definition of minors.  

\begin{definition}[Minors] \label{def:minors}
For any $N\times N$ matrix $A$ and $\mathbb T \subset \{1, \dots, N\}$, we define the minor of the first kind $A^{[\mathbb T]}$ as the $(N-|\mathbb T|)\times (N-|\mathbb T|)$ matrix with
\begin{equation*}
(A^{[\mathbb T]})_{ij} \;\deq\; A_{ij}, \quad i,j  \notin \mathbb T.
\end{equation*}
For any $N\times N$ invertible matrix $B$, we define the minor of the second kind $B^{(\mathbb T)}$ as the $(N-|\mathbb T|)\times (N-|\mathbb T|)$ matrix with
\begin{equation*}
(B^{(\mathbb T)})_{ij}=\left( (B^{-1})^{[\mathbb T]}\right)^{-1}_{ij}, \quad i,j  \notin \mathbb T,
\end{equation*} 
whenever $(B^{-1})^{[\mathbb T]}$ is invertible. Note that we keep the names of indices when defining the minors. By definition, for any sets $\mathbb U,\mathbb T\subset \{1,\dots, N\}$, we have 
\be\nonumber
(A^{[\mathbb T]})^{[\mathbb U]}=A^{[\mathbb T\cup \mathbb U]}, \quad
 (B^{(\mathbb T)})^{(\mathbb U)}=B^{(\mathbb T\cup \mathbb U)} .
\ee
For convenience, we shall also adopt the convention that for $i\in \mathbb T$ or $j\in \mathbb T$,
$$(A^{[\mathbb T]})_{ij}=0, \quad (B^{(\mathbb T)})_{ij}=0.$$
For $\mathbb T = \{a\}$ or $\mathbb T = \{a,b\}$, we shall abbreviate $(\{a\})\equiv (a)$ and $(\{a,b\})\equiv (ab)$. \end{definition}

\begin{remark}
In previous works, e.g. \cite{EKYY2,ErdYauYin2012Univ}, we have used the notation $(\cdot)$ for both the minor of the first kind and the minor of the second kind. Here we try to distinguish between $(\cdot)$ and $[\cdot]$ in order to be more rigorous.
\end{remark}

The following  identities were proved in Lemma 4.2 of \cite{ErdYauYin2012Univ} and Lemma 6.10 of \cite{EKYY2}.
\begin{lemma}[Resolvent identities]
For an invertible matrix $B\in \C^{N\times N}$ and $k\notin \{i,j\}$, we have
\begin{equation} \label{Gij Gijk}
B_{ij} \;=\; B_{ij}^{(k)} + \frac{B_{ik} B_{kj}}{B_{kk}}\, , \quad \frac1{B_{ii}} \;=\; \frac1{B_{ii}^{(k)}} -  \frac{B_{ik} B_{ki}}{B_{ii}^{(k)}B_{ii}B_{kk}}\, ,
\end{equation}
and
\begin{equation} \label{sq root formula2}
\frac 1{B_{ii}} \;=\;  (B^{-1})_{ii}-\sum_{k,l}^{(i)}(B^{-1})_{ik}  B^{(i)}_{kl} (B^{-1})_{li} .
\end{equation}
Moreover, for $i \neq j$ we have
\begin{equation} \label{sq root formula}
B_{ij} \;=\; - B_{ii} \sum_{k}^{(i)} (B^{-1})_{ik} B_{kj}^{(i)} \;=\; - B_{jj} \sum_k^{(j)}  B^{(j)} _{ik} (B^{-1})_{kj}\,.
\end{equation}
{The above equalities are understood to hold whenever the expressions in them make sense.}
\end{lemma}

Since the $N^\tau$ factor and the $N^{-D}$ bound for small probability event appear very often in our proof, we introduce the following notations.

\begin{definition} \label{whp}
For any non-negative $A$, we denote
 $$ \OO_\tau(A):= {\rm O}(N^{{\rm O}(\tau)} A)  .$$
We shall say an event $\mathcal E_N$ holds with high probability ({\it w.h.p.}) if for any fixed $D>0$,
 $$\P( \mathcal E_N)\ge 1-N^{-D}$$
for sufficiently large $N$. Moreover, we say $\mathcal E_N$ holds with high probability in $\Omega$ if for any fixed $D>0$,
 $$\P( \Omega\setminus \mathcal E_N)\le N^{-D}$$
for sufficiently large $N$.
 \end{definition}

The following lemma gives standard large deviation bounds that will be used in the proof of Lemma \ref{bneng}.

\begin{lemma}[Lemma 3.5 of \cite{isotropic}]\label{large_deviation}
Let $(X_i)$ be a family of  independent random variables and $(b_{i})$, $(B_{ij})$ be deterministic families of complex numbers, where $i,j=1, \ldots, N$. Suppose the entries $X_i$ satisfy $\mathbb E X_i=0$, $\mathbb E| X_i |^2 = 1$ and the bound(\ref{eqn:subgaus}). Then for any fixed $\tau>0$, we have
\begin{equation*}
\left|\sum_i {b_i X_i } \right| \le N^\tau\left( {\sum_i {\left| {b_i^{} } \right|^2 } } \right)^{1/2} , \quad \left|\sum_{i, j}\bar X_i B_{ij} X_j \right| \le N^\tau\left( {\sum_{i,j} {\left| B_{ij} \right|^2 } } \right)^{1/2},
\end{equation*}
 with high probability.
\end{lemma}

The following lemma provides estimates on the entries of $ (1-M^2  S )^{-1}$ and $ \left(1-S |M|^2\right)^{-1}S$. It will be used in the proof of Lemma \ref{bneng} and Theorem \ref{LLniu}, and its proof is delayed until Section \ref{matrix-norm}.

\begin{lemma}\label{maxnorm}
 Suppose that the assumptions for the strong form of Theorem \ref{LLniu}, i.e., \eqref{mouyzz}, \eqref{mouyzz2} and \eqref{jxw0}, hold.  If $\wt z$ satisfies 
$$ \re \wt z=e, \quad 0\le \im \wt z\le \im z, $$ 
then we have for $M\equiv M_\zeta^{\b g}( z, \wt z)$ and $S\equiv S_{\zeta}$,
 \be
 \left[ (1-M^2  S )^{-1} \right]_{ij} =
 \begin{cases}
 \delta_{ij}+{\rm O}( W^{-1}),\quad & {\rm  if} \quad |i-j|\le (\log N)^2W\\
 {\rm O}(N^{-c\log N}),\quad & {\rm  if} \quad |i-j|> (\log N)^2W
 \end{cases},
 \label{tianYz}
 \ee
and  
\be\label{zlende} 
\left\|\left(1-S|M|^2\right)^{-1}S\right\|_{\max}
= {\rm O} \left( \frac{1}{W\im z}+\frac{N}{W^2} \right). 
\ee
\end{lemma}

Now we can give the proof of Lemma \ref{bneng}.

\begin{proof}[Proof of Lemma  \ref{bneng}] \    The following proof is fairly  standard in random matrix theory and we will omit some details.  For simplicity, we drop $\zeta$ and $\b g$ in superscripts.  Using \eqref{sq root formula}, we have 
$
G_{ij}=-G_{ii}\sum_{k}^{(i)}H_{ik}G^{(i)}_{kj}
$ for $i\ne j$. 
Since the elements in $\{H_{ik}\}_{k=1}^N$ are independent of $G^{(i)}$, by the standard large deviations estimates in Lemma \ref{large_deviation}, 
we have that for any fixed $\tau>0$ and $D>0$, 
\be\label{miny}
\P\left(|G_{ij}|^2\le N^\tau |G_{ii}|^2  \sum_{k}S_{ik}|G^{(i)}_{kj}|^2\right)\ge1- N^{-D},\quad i\ne j .
 \ee
Since  $G_{ii}\asymp 1$ in $\Omega$,  \eqref{miny} implies that 
$$
\P\left({\bf 1}_{\Omega}|G_{ij}|^2 =\OO_\tau \left(\sum_{k}S_{ik}|G^{(i)}_{kj}|^2\right) \right)\ge1- N^{-D},\quad i\ne j .
$$
By  \eqref{Gij Gijk}, the definition of $T$ in \eqref{def: T}, and the bound for $T$ in \eqref{cllo}, we have 
 $$
 \sum_{k}S_{ik}|G^{(i)}_{kj}|^2\le 2 \sum_{k}S_{ik}|G _{kj}|^2+2\sum_{k}S_{ik}\frac{|G_{ki}G_{ij}|^2}{|G_{ii}|^2}= \OO (\Phi^2) 
 \quad \hbox{in}\; \Omega.
 $$
 Therefore, we obtain \eqref{34} for the $i\ne j$ case.  
 
 For the diagonal case, we define 
 $$
 \cal Z_i:= Q_i\left(\sum^{(i)}_{kl}H_{ik}H_{il}G^{(i)}_{kl}\right)-H_{ii}.
 $$
Using \eqref{sq root formula2}, \eqref{Gij Gijk}, the off-diagonal case for \eqref{34} we just proved, and the standard large deviations estimates in Lemma \ref{large_deviation}, 
we can get that for any fixed $\tau>0$, 
 $$
  \frac 1{G_{ii}} \;=\;  -z\mathds{1}_{i\in\llbracket 1,W\rrbracket} - \wt z\mathds{1}_{i\notin\llbracket 1,W\rrbracket}-{g}_i-\sum_{j} S_{ ij}G _{jj} 
 -\cal Z_i+\OO_\tau(\Phi^2)  , \quad \text{with }\cal Z_i =\OO_\tau \left( \Phi  \right),
 $$
holds with high probability in $\Omega$. With the definition of $M_i$ in \eqref{falvww}, we have 
 $$
 G^{-1}_{ii}-M_i^{-1}=-\sum_{j} S_{ ij}\left(G _{jj}-M_j\right)    +\OO_\tau(\Phi )  ,\quad \text{{\it w.h.p.} in }\Omega,
 $$ 
 which implies 
 $$
 M_i-G_{ii}=-\sum_{j} M_i^2 S_{ ij}\left(G _{jj}-M_j\right)    + \OO_\tau(\Phi )  + \OO \left(\max_i |G_{ii}-M_i|^2\right),\quad \text{{\it w.h.p.} in }\Omega.
 $$
 We rewrite the above estimate as
  $$
  \sum_j (1-M^2S)_{ij}\left(G _{jj}-M_j\right) =\OO_\tau(\Phi )  +\OO \left(\max_i |G_{ii}-M_i|^2\right).
  $$
Then with \eqref{tianYz} and the first bound in \eqref{cllo}, we can get \eqref{34} for the diagonal entries and complete the proof of Lemma \ref{bneng}
  \end{proof}
  
\subsection{The $T$-equation estimate.}\
A key component for the proof of Theorem \ref{LLniu} is the self-consistent equation for the $T$ variables. It leads to a self-improved bound on $\|G-M\|_{\max}$. This kind of approach was also used in \cite{ErdKnoYauYin2013-band} to prove a weak type delocalization result for random band matrices. To help the reader understand the proof, we first prove a weak $T$-equation estimate, i.e. Lemma \ref{YEniu1}, which will give the weak form of Theorem \ref{LLniu}. The stronger $T$-equation estimate will be stated in Lemma \ref{YEniu}, and its proof is put in the companion paper \cite{PartIII}.

\begin{lemma}[Weak $T$-equation estimate] \label{YEniu1} 
Under the assumptions of Theorem \ref{LLniu} (i.e., \eqref{mouyzz}, \eqref{mouyzz2}, \eqref{jxw0} and the assumption on $e$), the following statements hold provided  $\e_*>0$ is a sufficiently small constant.      
  Let  $\wt z$ satisfy 
  \be\label{xiazhou}
  \re \wt z=e, \quad N^{-10}\le \im \wt z\le \im z, 
   \ee
and $\Phi$ be any deterministic parameter satisfying 
$$
   W^{-1}\le   \Phi^2\le N^{- \delta} 
$$ 
for some fixed $\delta > 0$. Fix some $z$ and $\wt z$ (which can depend on $N$). If for any constants $\tau'>0$ and $D'>0$, 
  \begin{align}\label{GM11}
\mathbb P\left( \|G(z,\wt z) -M(z,\wt z)\|_{\max} \ge N^{\tau'}\Phi \right) \le N^{-D'},
    \end{align}
then for any fixed (small) $\tau>0$ and (large) $D>0$, we have
\be\label{GM21}
\mathbb P\left( \|T(z,\wt z) \|_{\max} \ge  N^\tau(\Phi_\#^{w})^2\right) \le N^{-D},  \quad  (\Phi_\#^{w})^2:= \left( \frac{N}{W\im z}+\frac{N^{2}}{W^2} \right) (\Phi^3+ N^{-1}) . 
\ee
 Furthermore, if the parameter $\Phi$  satisfies  
\be\label{46}
 \Phi \le  \min\left\{ \frac{W}{N^{1+\e_* + \e^*}} ,\frac{W^2}{N^{2+ \e^*}} \right\}, 
\ee
 then for any fixed $\tau>0$ and $D>0$ we have 
\be\label{44}
\|G(z,\wt z)-M(z,\wt z)\|_{\max}  \le  \Phi N^{ -\frac1{3}\e^*   } + N^{ \tau}\left(\frac{1}{\sqrt{W\im z}}+ \frac{N^{1/2}}{W}\right) 
\ee
with probability at least $1- N^{-D}$.
\end{lemma}

\begin{remark}\label{remarkY}
The above statements should be understood as follows. For any small constant $\tau>0$ and large constant $D>0$, (\ref{GM21}) and (\ref{44}) hold if (\ref{GM11}) holds for some constants $\tau',D'$ that depend on $\tau$ and $D$. In general, we need to take $\tau'<\tau$ to be sufficiently small and $D'>D$ to be sufficiently large. Compared with Lemma \ref{bneng}, we lose a much ``larger" portion of the probability set. Hence Lemma \ref{YEniu1} can only be iterated for ${\rm O}(1)$ number of times, while Lemma \ref{bneng} can be applied for ${\rm O}(N^C)$ times for any fixed $C>0$.
\end{remark}

\begin{proof}[Proof of Lemma \ref{YEniu1}] \ 
From the defining equation \eqref{def: T} of $T$,  we add and subtract $\sum_{k}S_{ik}|M_k|^2T_{kj} $ 
so that      
$$
T_{ij} 
=\sum_{k}S_{ik}|M_k|^2T_{kj} +\sum_{k}S_{ik}\left( |G_{kj}|^2-|M_k|^2T_{kj}\right).
$$
Therefore, we have
\begin{equation}\label{Tequation}
T_{ij}=\sum_{k}\left[\left(1-S|M|^2\right)^{-1}S\right]_{ik}\left(  |G_{kj}|^2-|M_k|^2T_{kj} \right).
\end{equation}
Isolating the diagonal terms,  we can write the $T$-equation as 
\be\label{T0}
T_{ij}=T_{ij}^{\,0}+\sum_{k\ne j }\left[(1-S|M|^2)^{-1}S\right]_{ik}\left(  |G_{kj}|^2-|M_k|^2T_{kj} \right)
,\quad T_{ij}^{\,0}:=\left[(1-S|M|^2)^{-1}S\right]_{ij}\left(  |G_{jj}|^2-|M_j|^2T_{jj} \right). 
\ee
By the definition of $T$, the assumption \eqref{GM11} and the estimate \eqref{bony} on $M_i$, we can get the simple bounds $G_{jj}={\rm O}(1)$ and $T_{jj}=\OO_\tau(\Phi^2)$. Applying these bounds to the definition of $T_{ij}^{\,0}$, we get 
\be\label{Tij0}
T_{ij}^{\,0}=\OO \left(\left[(1-S|M|^2)^{-1}S\right]_{ij}\right), 
\ee
which will be shown to be the main term of $T_{ij}$ up to an $N^\tau$ factor. By \eqref{zlende} and the condition \eqref{mouyzz2} on $\im z$, we have
\be\label{43}
\left[(1-S|M|^2)^{-1}S\right]_{ij}= {\rm O} \left( \frac{1}{W\im z}+\frac{N}{W^2} \right).
\ee
\begin{definition}[$\E_k$, $P_k$ and $Q_k$] We define $\E_k$ as the partial expectation with respect to the $k$-th row and column of $H$, i.e. $\E_k(\cdot) := \E(\cdot|H^{[k]})$.  For simplicity, we will also use the notations
\be\label{defPQ}
P_k :=\E_k , \quad Q_k := 1-\E_k.
\ee
 \end{definition}

Using this definition and the  bound \eqref{43}, we  rewrite   the off-diagonal  terms in \eqref{T0} into two parts: 
\be\label{damen}
\begin{split}
& \sum_{k\ne j }\left[(1-S|M|^2)^{-1}S\right]_{ik}\left(  |G_{kj}|^2-|M_k|^2T_{kj} \right) \\
& = \left( \frac{N}{W\im z}+\frac{N^2}{W^2} \right) 
 \left( \sum_{k\ne j} c_k\left(\mathbb E_k |G_{k j}|^2-|M_k|^2T_{k j}\right) + \sum_{k\ne j} c_k Q_k |G_{k  j}|^2  \right),
 \end{split}
\ee
where $c_k$ is a sequence of deterministic numbers satisfying 
$$
c_k:=\left[(1-S|M|^2)^{-1}S\right]_{ik} \left( \frac{N}{W\im z}+\frac{N^2}{W^2} \right)^{-1}={\rm O}(N^{-1}). 
$$
The following two lemmas provide estimates for the two parts in \eqref{damen}, where Lemma \ref{Qpart2} is a standard fluctuation averaging lemma.

\begin{lemma}\label{Ppart2} Suppose that $b_k$, $k\in \Z_N$, are deterministic coefficients satisfying $\max_k |b_k|={\rm O}(N^{-1})$.   Then under the assumptions of Lemma \ref{YEniu1}, we have that for any fixed (small) $\tau>0$, 
\be\label{eqn-Ppart2}
 \sum_{k\ne j} b_k\left(\mathbb E_k |G_{k j}|^2-|M_k|^2T_{k j}\right) 
=  \OO_\tau \left(\Phi^3  \right) ,\quad  j\in \Z_N, 
\ee
with high probability.
\end{lemma}

 \begin{proof}
By \eqref{sq root formula} and \eqref{GM11}, we have $-\sum_l^{(k)} H_{kl}G^{(k)}_{lj} = G_{kj}/ G_{kk} =\OO_\tau(\Phi)$ and $G_{kk}-M_k=\OO_\tau(\Phi)$ ({\it w.h.p.}). Then we can obtain that for $k \not = j$, 
\be\label{step1}
\mathbb E_k |G_{k j}|^2=\mathbb E_k  |M_k|^2 \left|\sum_l^{(k)} H_{kl}G^{(k)}_{lj}\right|^2 +\OO_\tau (\Phi^3)
= |M_k|^2 \sum_{l}^{(k)}s_{kl}\left|G^{(k  )}_{lj}\right|^2 + \OO_\tau (\Phi^3) 
\ee 
with high probability. Using \eqref{Gij Gijk}, we have 
$$G^{(k)}_{lj}=G _{lj}+\OO_\tau(|G_{lk}||G_{ kj }|)=G _{lj}+\OO_\tau(\Phi^2), \quad  l,j\ne k,
$$
with high probability. Inserting it into \eqref{step1} and using the definition \eqref{def: T}, we can obtain \eqref{eqn-Ppart2}.
\end{proof}

\begin{lemma}\label{Qpart2} Suppose that $b_k$, $k\in \bZ_N$ are deterministic coefficients satisfying $\max_k |b_k|={\rm O}(N^{-1})$. Then under the assumptions of Lemma  \ref{YEniu1}, we have for any fixed (large) $p\in 2\mathbb N$ and (small) $\tau>0$,  
\be\label{eqn-Qpart2}
\mathbb E\left|\sum_{k\ne j} b_k  Q_k |G_{kj}|^2\right|^p \le \left(N^{\tau} \Phi^{ 3} \right)^p,\quad  j\in \Z_N.
\ee
\end{lemma}
\begin{proof}
Our proof follows the arguments in \cite[Appendix B]{ErdKnoYauYin2013}. We consider the decomposition of the space of random variables using $P_k$ and $Q_k$ defined in \eqref{defPQ}. It is evident that $P_k$ and $Q_k$ are projections, $P_k+Q_k=1$, $P_kQ_k=0$, and all of these projections commute with each other. For a set $A\subset \Z_N$, we denote $P_A:=\prod_{k\in A} P_k$ and $Q_A:=\prod_{k\in A} Q_k$. Now fix any $j\in \Z_N$, we set $X_{k}:=Q_{k}|G_{k j}|^2$. Then for $p\in 2\N$, we can write
\begin{equation*}
\begin{split}
\mathbb E \left|\sum_{k\ne j} b_k X_k\right|^p = \sum_{k_1 ,k_2 , \ldots ,k_p }^* c_{\mathbf k} \mathbb E \prod_{s = 1}^p {X_{k_s } } & = \sum_{\mathbf k}^* c_{\mathbf k} {\mathbb E \prod_{s = 1}^p {\left( {\prod_{r = 1}^p {\left( {P_{k_r }  + Q_{k_r } } \right)} X_{k_s } } \right)}  } \\
& = \sum^*_{\mathbf k} c_{\mathbf k} {\sum\limits_{A_1 , \ldots ,A_p  \subset \left[ \mathbf k \right]} \mathbb E \prod_{s = 1}^p {\left( {P_{A_s^c } Q_{A_s} X_{k_s } } \right)} } ,
\end{split}
\end{equation*}
where $\mathbf k:=(k_1,k_2,\ldots, k_p)$, $[\mathbf k]:=\{k_1,k_2,\ldots,k_p\}$, $\sum^*$ means summation with indices not equal to $j$, and $c_{\mathbf k}$ are deterministic coefficients satisfying $c_{\mathbf k}={\rm O}(N^{-p}).$ Then with the same arguments as in \cite{ErdKnoYauYin2013} (more specifically, the ones between (B.21)-(B.24)), we see that to conclude \eqref{eqn-Qpart2}, it suffices to prove that for $k\in A \subset \Z_N \setminus \{j\}$ and any fixed $\tau>0$,
\begin{equation}
\left|Q_A X_k\right| =\OO_\tau\left( \Phi^{|A|+1}\right) \quad w.h.p. \label{claim}
\end{equation}

We first recall the following simple bound for partial expectations, which is proved in Lemma B.1 of \cite{ErdKnoYauYin2013}. 
Given a nonnegative random variable $X$ and a deterministic control parameter $\Psi$ such that $X\le \Psi$ with high probability. Suppose $\Psi\ge N^{-C}$ and $X\le N^C$ almost surely for some constant $C>0$. 
Then for any fixed $\tau>0$, we have
\be\label{partial_P}
\max_i P_i X =\OO_\tau (\Psi) \quad w.h.p.
\ee
In fact, \eqref{partial_P} follows from Markov's inequality, using high-moments estimates combined with the definition of high probability events in Definition \ref{whp} and Jensen's inequality for partial expectations. In the application to resolvent entries, the deterministic bound follows from $\|G\|\le (\im \wt z)^{-1} \le N^{10}$ by \eqref{xiazhou}.

Now the bound \eqref{claim} in the case $|A|=1$ follows from \eqref{partial_P} directly. For the case $|A|=n \ge 2$, we assume without loss of generality that $j=1$, $k=2$ and $A=\{2 , \ldots, n+1\}$. It suffices to prove that
\begin{equation}
Q_{n+1}\cdots Q_3 |G_{21}|^2 =\OO_\tau\left( \Phi^{n+1}\right) .\label{claim2}
\end{equation}
Using the identity \eqref{Gij Gijk}, we can write
$$Q_3 |G_{21}|^2 = Q_3 \left(G_{21}^{(3)}+\frac{G_{23}G_{31}}{G_{33}}\right)\overline{\left(G_{21}^{(3)}+\frac{G_{23}G_{31}}{G_{33}}\right)} = Q_3 \left( \overline{G_{21}^{(3)}} \frac{G_{23}G_{31}}{G_{33}}+ {G_{21}^{(3)}} \overline{\frac{G_{23}G_{31}}{\overline{G_{33}}} }+ \left|\frac{G_{23}G_{31}}{G_{33}}\right|^2\right).$$
Note that the leading term $Q_3 \left|G_{21}^{(3)}\right|^2$ vanishes since $G_{21}^{(3)}$ is independent of the 3rd row and column of $H$, and the rest of the three terms have at least three off-diagonal resolvent entries. We now act $Q_4$ on these terms, apply (\ref{Gij Gijk}) with $k=4$ to each resolvent entry, and multiply everything out. This gives a sum of fractions, where all the entries in the numerator are off-diagonal and all the entries in the denominator are diagonal. Moreover, the leading order terms vanish,
$$Q_4Q_3 \left( \overline{G_{21}^{(34)}} \frac{G_{23}^{(4)}G_{31}^{(4)}}{G_{33}^{(4)}} + {G_{21}^{(34)}}  \overline{\frac{G_{23}^{(4)}G_{31}^{(4)}}{\overline{G_{33}^{(4)}}}}\right)=0,$$
and each of the surviving term has at least four off-diagonal resolvent entries. We then continue in this manner, and at each step the number of off-diagonal resolvent entries in the numerator increases at least by one. Finally, 
$Q_{n+1}\cdots Q_3 |G_{kj}|^2$ is a sum of fractions where each of them contains at least $n+1$ off-diagonal entries in the numerator. Together with \eqref{partial_P}, this gives the estimate \eqref{claim2}, which further proves \eqref{claim}.
\end{proof}

\begin{remark}\label{33}   Lemma \ref{Qpart2} asserts that the $Q_k$ operation yields an improvement by a factor $\Phi$. In fact, for the regular resolvents of band matrices, a stronger version of averaging fluctuation results was proved in \cite{EKY_Average}. We believe that following the methods there, the bounds in Lemma \ref{Ppart2} and Lemma \ref{Qpart2} can be improved to 
\be\label{semi_strong}
\OO_\tau \left( \Phi^{ 4  }+   W^{-1/2}\Phi^2 \right).
\ee
In this paper, however, we will skip the discussion on the strategy in \cite{EKY_Average}, since its proof is rather involved, and more importantly, we will prove an even stronger bound, i.e., \eqref{GM2} below, in Part III of this series \cite{PartIII}. With \eqref{semi_strong}, the $\Phi_\#^{w}$ in \eqref{GM21} can be improved to 
$$
(\Phi_\#^{w})^2 = \left( \frac{N}{W\im z}+\frac{N^{2 }}{W^2} \right) ( \Phi^4 +W^{-1/2}\Phi^2+ {N}^{-1}), 
$$
and the condition \eqref{46} becomes 
\be\label{462}
  \Phi^2 \le  \min\left\{ \frac{W}{N^{1+\e_* + \e^*}} ,\frac{W^2}{N^{2+ \e^*}} \right\} .
\ee
Using this estimate, the conditions \eqref{67} can be weaken to 
\be\log _N W \ge \max\left\{ \frac45 + \epsilon^* , \ \frac23 + \frac23\e_*+ \e^* \right\} .\label{inbetween}\ee
\end{remark}

Now we finish the proof of Lemma \ref{YEniu1}. Using \eqref{damen}, Lemma \ref{Ppart2}, Lemma \ref{Qpart2} and Markov's inequality, we can get that
$$
\sum_{k\ne j}\left[(1-S|M|^2)^{-1}S\right]_{ik}\left(  |G_{kj}|^2-|M_k|^2T_{kj} \right)= \OO_\tau \left(\left( \frac{N}{W\im z}+\frac{N^{2}}{W^2} \right) \Phi^3 \right)
$$ 
with high probability. Note that it only includes the off-diagonal  terms, i.e. $k\ne j$ terms. Now plugging it into the $T$-equation \eqref{T0} and using \eqref{Tij0}, we obtain \eqref{GM21}. 

Finally, we need to prove \eqref{44}. Clearly, if \eqref{46} holds, then $\Phi \le N^{-\delta}$ and $(\Phi_\#^w)^2 \nc\le N^{-2 \delta}$ for some constant $\delta>0$. Thus \eqref{cllo} is satisfied, and then \eqref{44} follows from an application of (\ref{GM21}) and Lemma \ref{bneng}. This completes the proof of Lemma \ref{YEniu1}.   
 \end{proof}

The following lemma gives a stronger form of Lemma \ref{YEniu1}. It will be proved in the companion paper \cite{PartIII}. Here we recall the notation in \eqref{tri}.

 \begin{lemma}[Strong $T$-equation estimate] \label{YEniu} 
 Suppose the assumptions of Theorem \ref{LLniu} (i.e., \eqref{mouyzz}, \eqref{mouyzz2}, \eqref{jxw0} and the assumption on $e$) and \eqref{xiazhou} hold. Let  $\Phi$ and  $\wt\Phi$ be  deterministic parameters satisfying   
  \be\label{GM1.5}
   W^{-1}\le \wt\Phi^2\le  \Phi^2\le  \wt \Phi \le N^{-\delta}
 \ee
for some constant $\delta>0$. Fix some $z$ and $\wt z$ (which can depend on $N$). If for any constants $\tau'>0$ and $D'>0$,
  \begin{align*}
    \P\Big( \|G(z,\wt z) -M(z,\wt z)\|_{\max} \ge N^{\tau'}\Phi\Big)+   \P\left(   \tnorm{G}^2(z,\wt z) \ge N^{1+\tau'} \wt \Phi^2\right)
    \le N^{-D'},
    \end{align*}
 then for any fixed (small) $\tau>0$ and (large) $D>0$, we have 
\begin{equation}\label{GM2}
   \P\left( \|T(z,\wt z) \|_{\max}\ge  N^\tau \Phi_\#^2  \right)\le  N^{-D},\quad 
  \Phi_\#^2:= \left( \frac{N}{W\im z}+\frac{N^{2}}{W^2} \right) \left(\Phi^2 {  \wt \Phi^2}+\Phi^2 {  N^{-1/2}}+N^{-1}\right)   .
\end{equation}
Furthermore, if the parameter $\wt \Phi$  satisfies  
\be\label{lakezz}
 \wt  \Phi^2 \le  \min\left\{ \frac{W}{N^{1+\e_* + \e^*}} ,\frac{W^2}{N^{2+ \e^*}} \right\},
\ee
then for any fixed $\tau>0$ and $D>0$ we have 
\be\label{44st}
\|G(z,\wt z)-M(z,\wt z)\|_{\max}   \le  \Phi N^{ -\frac1{3}\e^*   } + N^{ \tau}\left(\frac{1}{\sqrt{W\im z}}+ \frac{N^{1/2}}{W}\right)  
\ee
with probability at least $1- N^{-D}$.
 \nc \end{lemma}

The Remark \ref{remarkY} also applies to this lemma. 
Note that \eqref{44} or \eqref{44st} gives a self-improved bound on $\|G-M\|_{\max}$, which explains how we can improve the estimate on $G$ (from $\Phi$ to $\Phi_\#)$ via $T$ equations. As long as we have an initial estimate such that \eqref{46} or \eqref{lakezz} holds, we can then iterate the proof and improve the estimate on $G$ to $\Phi_{{\rm goal}}=\left(\frac{1}{\sqrt{W\im z}}+ \frac{N^{1/2}}{W}\right) $ in \eqref{jxw}.

\begin{proof}[Proof of Lemma \ref{YEniu}]
See the proof of Theorem 2.7 in part III of this series \cite{PartIII}.
\end{proof}

\section{Proof of  Theorem \ref{LLniu}}

Fix a parameter $0<\e_0<\e_*/5$. We define 
$$
 \wt z_n:= \re \wt z+ \ii N^{-n\e_0 }\im   z,
$$
so that $\im \wt z_{n+1}=N^{-\e_0}\im \wt z_{n}$. The basic idea in proving Theorem \ref{LLniu} is to use mathematical induction on  $n\in \N$.

The proofs of the weak form and strong form of Theorem \ref{LLniu} are completely parallel. In the following proof, we will only remark on the minor differences between them. 

\medskip \noindent 
{\bf Step 0}: The special case with $\wt z=z$ and  $\zeta=0$, $\b g=\b 0$ (i.e. $G(H,z)$ is the ordinary resolvent of a generalized Wigner matrix) was proved in \cite{ErdKnoYauYin2013}.  The proof given there can be carried over to our case without changes under the assumptions of Theorem \ref{LLniu} when $\wt z=z$ and $\im z\ge W^{-1+\delta}$ for some fixed $\delta>0$.   

This gives that 
$$
\P\left( \|G(z, z)-M( z, z)\| _{\max}\ge  \frac{N^{\tau}}{\sqrt{W\im z}}\right)\le N^{-D},
$$ 
 for any fixed $\tau>0$.
This bound is clearly stronger than the one in \eqref{jxw}.

\medskip 

\noindent   {\bf Step 1}: Consider  the case $n=0$, i.e., $G(z, \wt z_0)$, where we have
$$\re \wt z_0=\re \wt z,\quad \im \wt z_0=\im z.$$ 
We claim that for any $w,\wt w\in \C_+$,
\be\label{mnn}
\|G(w,\wt w)\|_{L^2\to L^2}\le \frac{1}{\min (\im w, \im \wt w)}.
\ee
To prove it, we first assume that $\im w = a + \im \tilde w$ with $a \ge 0$. We write 
$$
G(w,\wt w) = (A - \ii a J - \ii\im \tilde w )^{-1} ,  \quad J_{kl}={\bf 1}_{ k\in \llbracket 1, W\rrbracket} \delta_{kl} ,
$$ 
where $A$ is a symmetric matrix. Then 
$$
(A - \ii a J - \ii\im \tilde w )^* (A - \ii a J - \ii\im \tilde w ) = (A - \ii a J  )^* (A - \ii a J )  + 2 a(\im \tilde w) J  + (\im \tilde w)^2 \ge (\im \tilde w)^2.
$$
Obviously, we have a similar estimate with $\im \tilde w$ replaced by $\im w$ when $\im w\le \im \tilde w$. This proves the claim \eqref{mnn}.

Now by the definition of $T$ and \eqref{bandcw1}, we know
$$
  |T_{ij}{ (z, \wt z_0)}|\le \frac{C_s}{W}\sum_{k}|G_{kj}(z,\wt z_0)|^2 = \frac{C_s \im G_{jj}(z,\wt z_0)}{W\im z},
$$
where in the second step we used the so-called Ward identity that for any symmetric matrix $A$ and $\eta>0$,
\be\label{ward}\sum_k |R_{kj}(A,\ii\eta)|^2 = \frac{\im R_{jj}(A,\ii\eta)}{\eta},\quad R(A,\ii\eta):=(A-\ii\eta)^{-1}.
\ee
Obviously, the same argument gives that
\be\label{123}
\|T{ (z, \wt z_0(t))}\|_{\max}\le \frac{C_s \max_j \im G_{jj}(z,\wt z_0(t))}{W\im z}, \quad \wt z_0(t):= (1-t)z+t\wt z_0, \ \ t\in [0,1].
\ee
Now we claim that for any small enough $\tau>0$,
\be\label{1234}
\sup_{s\in [0,1]}  \P\left(\|G(z,\wt z_0(t))-M(z,\wt z_0(t))\|_{\max}\ge \frac{ N^{\tau}}{\sqrt{W\im z}}\right)\le N^{-D}.
\ee
\nc
To prove \eqref{1234}, we first note that for any $w, w'\in \mathbb C$, 
\be\label{coz}
  G(z, w)=G(z, w')+  G(z, w)(w-w')\wt J G(z, w'), 
  \quad \wt J_{kl}={\bf 1}_{k \notin \llbracket 1, W\rrbracket} \delta_{kl}.
\ee
This implies that 
$$
\|\partial_{\wt z}G  (z,\wt z)\|_{\max}\le   \sqrt N \|G(z,\wt z)\|_{L^2\to L^2}  \|G(z,\wt z)\|_{\max}   \le      \frac{  \sqrt N }{\min (\im z, \im \wt z)} \|G\| _{\max}.
$$
In particular, in this step we have 
\be \label{continu}
 \|\partial_{s}G  (z, \wt z_0(t))\|_{\max}\le    C N^{1/2+\e_*}|z-\wt z_0| \|G(z, \wt z_0(t))\| _{\max} . 
\ee
This provides some continuity estimate on $G  (z, \wt z_0(t))$, which shows that \eqref{1234} can be obtained from the following estimate: 
\be \label{zwjj2}
 \max_{k\in \llbracket0,N^5\rrbracket}\P\left( \left\|G (z, \wt z_0(kN^{-5}))-M ( z, \wt z_0(kN^{-5}))\right\| _{\max}\ge  \frac{N^{\tau}}{\sqrt{W\im z}} \right)\le N^{-D}.
 \ee 
 From Step 0,  this estimate  holds for $k=0$. By  induction, we assume that \eqref{zwjj2} holds for $k=k_0$.  Then using \eqref{continu} and \eqref{dozy}, we know that the first estimate of \eqref{cllo} holds for $G(z, \wt z_0(t))$ with $t=(k_0+1)N^{-5}$.  Then by \eqref{123} and applying Lemma \ref{bneng}, we obtain \eqref{zwjj2} for $k=k_0+1$. This completes the proof of \eqref{zwjj2} and \eqref{1234}. Note that the estimate \eqref{1234} applied to $G(z, \wt z_0(1))$ is the result we want for this step

 \medskip 
 \noindent 
  {\bf Step 2}:  
Suppose that for some $n\in \N$ with $\im \wt z_n\ge N^{-10}$, \eqref{jxw} holds for  $G(z, \wt z_n)$  and $M(z, \wt z_n)$ for any large $D>0$. We first prove the following estimate for $G(z, \wt z_{n+1})-M(z, \wt z_{n+1})$, which is weaker than \eqref{jxw}:   
  \be\label{jxw23}
 \P\left(\|G(z,\wt z_{n+1})-M(z,\wt z_{n+1} )\|_{\max}\ge  N^{ \tau}\left(\frac{N^{1/2+\epsilon_0}}{W\sqrt{\im z}}+ \frac{N^{1+\epsilon_0}}{W^{3/2}} \right) \right)\le N^{-D}
  \ee   
  for any fixed $\tau>0$.
  
For any $w, w'\in \mathbb C^+$ satisfying
\be\label{327}
  \re w = \re w', \quad N^{-\e_0}  \im w'  \le  \im w \le \im w',
\ee
using \eqref{327} and \eqref{coz}, we have 
  \begin{align*}
\sum_i  |G_{ij}(z,w)|^2  & \le 2\Big(1+  |w-w'|^2
  \left\|G(z, w)\right\|^2_{L^2\to L^2}\Big)\sum_{i}|G_{ij}(z, w')|^2  \notag \\
  & \le 2\Big (1+  \frac {  (\im w')^2} { (\im w)^2} \Big)  \sum_{i}|G_{ij}(z, w')|^2 \le  3N^{2\e_0}\tnorm{G (z, w')}^2, 
\end{align*}
where we have used \eqref{mnn} to bound  $\left\|G(z, w)\right\|^2_{L^2\to L^2}$.
We apply this inequality with $w' = \wt z_n$ and $w$ satisfying \eqref{327}.
Using \eqref{jxw} and the definition \eqref{tri}, we can bound $ \tnorm{G (z, \wt z_n)}^2$ as  
\be\label{bdll}
\sup_{\substack{ \re w = \re  \wt z_n,  \\ \im  \wt z_{n+1} \le  \im w \le \im  \wt z_n }}  \|T(z,  w)\|_{\max}  \le  
\sup_{\substack{ \re w = \re  \wt z_n,  \\ \im  \wt z_{n+1} \le  \im w \le \im  \wt z_n }}   \frac{C}W \tnorm{G (z, w) }^2 =\OO_\tau\left(\frac{N^{1+2\e_0}}{W^2\im z}+ \frac{N^{2+2\e_0}}{W^3}\right)
\ee
with high probability for any fixed $\tau>0$.

We now consider interpolation between $\wt z_n$ and $\wt z_{n+1}$: 
 $$
 \wt z_{n,m}=\wt z_n- \ii  (\im\wt z_n-\im \wt z_{n+1})mN^{-50}, \quad   m\in \llbracket 0, N^{50}\rrbracket. 
 $$
We would like to use Lemma \ref{bneng} and induction to prove that \eqref{jxw23} holds for $ G(z, \wt z_{n,m})-M(z, \wt z_{n,m})$ for all $m$. First, we know \eqref{jxw23} holds for $G(z, \wt z_n)$. Then suppose \eqref{jxw23} holds for  
$G(z, \wt z_{n,j})$ for all $j \le m-1$.  We now  verify that \eqref{cllo} holds for 
$G(z, \wt z_{n,m})$ with $\Phi^2=N^\tau \Phi^2_0$ for any fixed $\tau>0$, where
$$
\Phi^2_0:= \frac{N^{1+2\e_0}}{W^2\im z}+ \frac{N^{2+2\e_0}}{W^3}.
$$  
To this end, we note that \eqref{bdll} already verifies the bound on $\| T(z, \wt z_{n,m}) \|_{\max}$
in \eqref{cllo} for all $ m\in \llbracket 0, N^{50}\rrbracket$.
By using $\|\partial_{\wt z\, }G \|_{\max}\le N \|G \|^2_{\max}$ (which follows from \eqref{coz}), \eqref{dozy}, $|\wt z_{n,m-1} - \wt z_{n,m}| \le N^{-50}$, and \eqref{bdll} (to bound $\|G \|^2_{\max}$ by $\tnorm{G}^2$), we note that for sufficiently small constant $\delta>0$,
$$
  \|G(z,\wt z_{n,m-1})-M(z,\wt z_{n,m-1})\|_{\max}\le N^{-2\delta}\implies   \|G(z,\wt z_{n,{m}})-M(z,\wt z_{n,m})\|_{\max}\le N^{-\delta}.
$$
This proves  the first bound in \eqref{cllo}  for $G(z,\wt z_{n,{m}})$. 
Then Lemma  \ref{bneng} asserts that \eqref{34} holds for $G(z, \wt z_{n,m})$ with $N^\tau\Phi_0$ for any fixed $\tau>0$.  This proves \eqref{jxw23} (i.e. the $m=N^{50}$ case) by induction.

\medskip \noindent
  {\bf Step 3}:  Suppose that for some $n\in \N$ with $\im \wt z_n\ge N^{-10}$, \eqref{jxw} holds for  $G(z, \wt z_n)$  and $M(z, \wt z_n)$ for any large $D>0$. 
We  have proved that \eqref{jxw23} and \eqref{bdll} hold for $G(z, \wt z_{n+1})$. 
We now apply Lemma \ref{YEniu1} to  prove the weak form of Theorem \ref{LLniu}. First, the condition \eqref{GM11} holds with $\Phi=\frac{N^{1/2+\epsilon_0}}{W\sqrt{\im z}}+ \frac{N^{1+\epsilon_0}}{W^{3/2}} $. In order for the condition \eqref{46} to hold,
we need
\be\label{344}
\frac{N^{1/2+\epsilon_0}}{W\sqrt{\im z}}+ \frac{N^{1+\epsilon_0}}{W^{3/2}}  \le  \min\left\{ \frac{W}{N^{1+\e_* + \e^*}} ,\frac{W^2}{N^{2+ \e^*}} \right\}  ,
\ee
which is satisfied if  
$$
 W \ge 2\max\left(N^{\frac67 + \frac27 \e_0+\frac27\e^* } ,N^{\frac34 + \frac34\e_*+ \frac12 \e_0+\frac12\e^* } \right).
$$
If we take $\epsilon_0< \epsilon^*$, \eqref{GM11} implies \eqref{44} under the condition \eqref{67}. We then apply Lemma \ref{YEniu1} again, and after at most $3/\e^*$ iterations we obtain that 
\be\label{mainPhi}
 \|G(z,\wt z_{n+1})-M(z,\wt z_{n+1})\|_{\max} \le N^{ \tau}\left(\frac{1}{\sqrt{W\im z}}+ \frac{N^{1/2}}{W}\right). 
\ee
By induction on $n$ (with the number of inductions $\le 10/\e_0$),  the main estimate \eqref{mainPhi} for $G(z, \wt z_{n})$ holds for all $n$ as long as   $\im \wt z_{n} \ge N^{-10}$. 

\medskip

Similarly,  we can apply Lemma \ref{YEniu} to  prove the strong  form of Theorem \ref{LLniu}. 
As in the previous argument, \eqref{jxw23} and \eqref{bdll} hold for $G(z, \wt z_{n+1})$ assuming  \eqref{jxw} for $G(z, \wt z_n)$ and $\im \wt z_n\ge N^{-10}$. Therefore,  we can choose $\Phi$ and 
$\wt \Phi$ as  
$$
 \Phi= \frac{N^{1/2+\epsilon_0}}{W\sqrt{\im z}}+ \frac{N^{1+\epsilon_0}}{W^{3/2}} ,
 \quad 
 \wt \Phi=\frac{N^{\e_0}}{\sqrt{W\im z}}+ \frac{N^{1/2+\e_0}}{W},
 $$
where the choice of $\wt \Phi$ follows from using \eqref{bdll}. It is easy to see that \eqref{GM1.5} holds.
In order to apply Lemma \ref{YEniu}, we need \eqref{lakezz}, i.e., 
\be\nonumber
\left( \frac{N^{\e_0}}{\sqrt{W\im z}}+ \frac{N^{1/2+\e_0}}{W}\right)^2 \le  \min\left\{ \frac{W}{N^{1+\e_* + \e^*}} ,\frac{W^2}{N^{2+ \e^*}} \right\} ,
\ee  
which is satisfied if
$$ W \ge 2\max\left(N^{\frac34 + \frac12 \e_0+\frac14\e^* } ,N^{\frac12 + \e_*+ \e_0+\frac12\e^* } \right).$$
Clearly, the assumption \eqref{jxw0} guarantees this condition if we choose $\e_0<\epsilon^*/2$. Again, we can apply Lemma \ref{YEniu} iteratively until we get \eqref{mainPhi} for $G(z, \wt z_{n+1})$. The rest of the proof for the strong form of Theorem \ref{LLniu} is the same as 
the proof for the weak form.

 \medskip \noindent {\bf Step 4}:  We now prove \eqref{jxw} for $G(z,\wt z)$ with $\im \wt z = 0$ by using continuity from the estimate for $G(z,\wt z)$ with $\im \wt z = N^{-10}$ established in {\bf Step 3}. It is easy to see that 
\be\label{continuG}
\partial_{\wt z\, } \|G(z,\wt z\, )\|_{\max} \le   \|\partial_{\wt z\, }G(z,\wt z\, )\|_{\max}\le N \|G(z,\wt z\, )\|^2_{\max}.
\ee
With \eqref{continuG} and using \eqref{mainPhi} for $G(z, \re\wt z+\ii N^{-10})$, we can obtain that 
$$\sup_{0 \le \eta \le N^{-10}}\|G(z,\re \wt z+\ii \eta)\|_{\max}={\rm O}(1), \quad w.h.p.$$
Then using \eqref{dozy}, \eqref{coz} and \eqref{mainPhi} for $G(z, \re\wt z + \ii N^{-10})$, we obtain that \eqref{jxw} holds for $G(z, \re\wt z)$.

\begin{remark} 
If we use the bound in Remark \ref{33} and the condition \eqref{462} instead of \eqref{46},  then the restriction \eqref{344}  
becomes 
$$\left(\frac{N^{1/2+\epsilon_0}}{W\sqrt{\im z}}+ \frac{N^{1+\epsilon_0}}{W^{3/2}}  \right)^2 \le  \min\left\{ \frac{W}{N^{1+\e_* + \e^*}} ,\frac{W^2}{N^{2+ \e^*}} \right\} $$ 
which gives restriction in \eqref{inbetween}. 
So we get a result in between the weak and strong forms of   Theorem \ref{LLniu}.
\end{remark}

 \section{Properties of $M$} \label{matrix-norm} 
 
The main goal of this section is to derive some deterministic estimates related to $(M_\zeta^{\b g})_i$, $i\in \Z_N$. In particular, we will finish the proof of Lemma \ref{UE} and Lemma \ref{maxnorm}.

\subsection{The stability.}\
The system of self-consistent equations \eqref{falvww} is a perturbation of the standard self-consistent equation 
$$m_{\rm sc}^{-1}=-\wt z - m_{\rm sc}$$
for $m_{\rm sc}(\wt z)$. Thus our basic strategy is to use the standard perturbation theory (see \eqref{mGYY} below) combined with a stability estimate for the self-consistent equation (i.e. the operator bound \eqref{dust}). We first recall the following elementary properties of $m_{\rm sc}$, which can be proved directly using \eqref{msc}.

\begin{lemma} \label{lem: msc}
 We have for all $z=E+\ii\eta$ with $\eta > 0$ that
$$
   |m_{\rm sc}(z)| = |m_{\rm sc}(z)+z|^{-1}\le 1.
$$
Furthermore, there is a constant $c > 0$ such that for $E \in [-10, 10]$ and $\eta \in (0, 10]$ we have
\begin{equation} \label{m is bounded}
c \;\leq\; \abs{m_{\rm sc}(z)} \;\leq\; 1 - c\eta\,,
\end{equation}
\begin{equation} \label{m'}
\left|\partial_z m_{\rm sc}(z)\right| \;\leq\; c^{-1}(\kappa+\eta)^{-1/2},
\end{equation} 
$$
\abs{1-m^2_{\rm sc}(z)} \;\asymp\; \sqrt{\kappa + \eta} \,,
$$
as well as
$$
\im m_{\rm sc}(z) \;\asymp\;
\begin{cases}
\sqrt{\kappa + \eta} & \text{if $\abs{E} \leq 2$}
\\
\frac{\eta}{\sqrt{\kappa + \eta}} & \text{if $\abs{E} \geq 2$}
\end{cases},
$$
where $\kappa:= \big| |E|-2\big|$ denotes the distance of $E$ to the
spectral edges. 
\end{lemma}

The following lemma will be used in the proof of Lemma \ref{UE} and Lemma \ref{maxnorm}. Recall that $S_0$ is the matrix with entries $s_{ij}$, which is defined in Definition \ref{defHg}.

\begin{lemma} \label{lem_m0}
Assume $|\re \wt z|\le 2-\kappa$ for some constant $\kappa>0$ and denote $m=m_{\rm sc}(\wt z+\ii 0^+ )$. Then for any fixed $\tau> 0$, there exist constants $ c_1, C_1>0$ such that
\be\label{gbzz2}
\left \|\left (\frac{m^2S_0+\tau}{1+\tau}\right)^2\right \|_{L^\infty\to L^\infty}<  1-c_1.
\ee
Furthermore, 
\be\label{dust}
\left\|(1-m^2S_0)^{-1}\right\|_{L^\infty\to L^\infty}\le C_1.
\ee
\end{lemma}

\begin{proof}
For some small constant $\tau>0$ we write 
\be\label{gbzz}
(1-m^2S_0)^{-1}= \frac{1}{1+\tau} \sum^\infty_{ k=0 } \left(\frac{m^2S_0+\tau}{1+\tau}\right)^k.
\ee
Assuming  \eqref{gbzz2}, we get that  
$$
\|(1-m^2S_0)^{-1}\|_{L^\infty\to L^\infty}
\le \frac 1 {1+\tau}     \left (  1+ \left\|\frac{m^2S_0+\tau}{1+\tau}\right\|_{L^\infty\to L^\infty}\right)  \sum_{j=0}^\infty   \left\| \frac{m^2S_0+\tau}{1+\tau}\right\|_{L^\infty\to L^\infty}^{2j} 
\le C_1, 
$$
which proves \eqref{dust}. 

We now prove  \eqref{gbzz2}. Suppose that  there is a vector  $\bv\in \C^N$ so that $\|\bv\|_\infty=1$ and
 $$\left |\frac{[(m^2S_0+\tau)^2 \bv ]_i}{(1+\tau)^2} \right | = 1 - \e$$ 
 for some $i\in \Z_N$ and $\epsilon\equiv \e_N \to 0^+$. Hence 
\be\label{333} 
 (1+ 2 \tau + \tau^2) (1- \epsilon)  = \big | m^4 b + 2 \tau m^2 a  + \tau^2 v_i \big | \le 
 |  b| + 2 \tau |a| + \tau^2 | v_i | \le 1+ 2 \tau + \tau^2, 
\ee
where $a:=(S_0 \bv)_i,$ $b:=(S_0^2\bv)_i$ and we have used the bounds $|m|\le 1$, $ |a| \le 1$ and  $ |b| \le 1$ (since $\|S_0\|_{L^\infty \to L^\infty}=1$). It will be clear that the $|m|=1$ case is most difficult and we will 
assume this condition in the following proof. Moreover, we assume with loss of generality that $v_i > 0$ (by changing the global phase of $\mathbf v$). Now $m$, $a$ and $b$ are complex numbers, and the inequality \eqref{333} implies that $m^4 b$, $m^2 a$ and $ v_i $ have almost the same phases. Since $|v_i| \le 1$, $ |b| \le 1$ and  $ |a| \le 1$, \eqref{333} implies that for some constant $C>0$ independent of $\e$, 
\be \label{ab}
 v_i\ge 1- C\e, \quad |b- m^{-4}| \le C \e , \quad | a - 
 m^{-2} | \le C \e .
\ee
  \nc 
Since $m$ is a unit modulus complex number with imaginary part of order 1, we have that $\delta:=|m^{-2}-m^{-4}| $ is a number of order 1 and 
$$ |a-b|>\delta/2.$$ 
Fix the index $i$ and denote $c_j := (S_0)_{ij}$, $d_j: = (S_0^2)_{ij}$. Then $\sum_j c_j = 1 = \sum_j d_j$. Hence \eqref{ab} implies  
$$
1  - {\rm O}(\e) = \re (a \bar a )= \sum_j c_j   \re (v_j   \bar a), \quad  1  - {\rm O}(\e) = \re (b \bar b )= \sum_j d_j \re (v_j \bar b) , 
$$
 where ${\rm O}(\e) $ denote a positive number bounded by $C \e$ for some constant $C> 0$ independent of $\e$. For any $0<r <1$, denote by $A_r: = \{ j:  \re (v_j   \bar a) \ge 1 - r \}$ and let $\alpha_r: = \sum_{j\in A_r}  c_j$. 
Then we have 
\be\nonumber
  \alpha_r  \ge \sum_{j\in A_r}  c_j   \re (v_j   \bar a)  =  1 - {\rm O}(\e)- \sum_{j\not \in A_r}  c_j   \re (v_j   \bar a) \ge 1 - {\rm O}(\e) -   (1-\alpha_r)  (1-r) = \alpha_r + r - \alpha_r r- {\rm O}(\e),
\ee
which implies that 
\be\label{zjmmus}
\sum_{j\in A_r}  c_j = \alpha_r  \ge   1 - {\rm O}(\e)r^{-1}, \quad \sum_{j\notin A_r}  c_j =   {\rm O}(\e)r^{-1}.
\ee
Similarly, if we define $B_r := \{ j:  \re (v_j   \bar b) \ge 1 - r \}$, then 
 \be\label{zjmmus2} 
\sum_{j\notin B_r}  d_j= {\rm O}(\e)r^{-1}.
\ee   
We claim that if $r \ge C \e$ for some large enough constant $C>0$, then $A_r \cap B_r \not= \emptyset$. To see this, we define $U: = \{j: |i-j| \le  W\}$.  By  \eqref{bandcw1} and the definition of $c_j$, we have $c_j \ge 
c_s W^{-1}$ for $j\in U$.  Clearly, we also have $d_j \ge \frac12 c_s W^{-1}$ for $j \in U$.  Then with \eqref{zjmmus} and \eqref{zjmmus2}, we have
$$
\#\{j\in U\setminus A_r\} =  {\rm O}(\e)r^{-1}c_s^{-1}W,\quad \#\{j\in U\setminus B_r\} =  {\rm O}(\e)r^{-1}c_s^{-1}W.
$$
If we choose  $r= C \e$ for some large enough constant $C>0$, then the above two inequalities imply 
$A_r \cap B_r \not= \emptyset$, since $|U|=W$.  \nc  
Thus there is an index $j$ such that 
\be\label{abj}
\re (v_j   \bar a) \ge 1 - r , \quad \re (v_j   \bar b) \ge 1 - r.
\ee 
Since $|a|\le 1$, $|b|\le 1$, $|v_j|\le 1$ and $|a-b|>\delta/2$, \eqref{abj} is possible only if $r \gtrsim \delta$, which contradicts the fact that $r\to 0$ when $\e\to 0$. This proves \eqref{gbzz2}.  
\end{proof}

\subsection{Proof of Lemma \ref{UE}.}\ 
With Lemma \ref{lem_m0}, we can now give the proof of Lemma \ref{UE}.
\begin{proof}[Proof of Lemma \ref{UE}] 
We first prove the existence and continuity of the solutions to \eqref{falvww}. The proof is a standard application of the contraction principle. Denote by $\b z: = (z_1, \ldots, z_N)$, $\b x: = (x_1, \ldots, x_N)$ and $\mathbf M:=((M_\zeta^{\b g})_1, \ldots, (M_\zeta^{\b g})_N)$ with 
$$
z_i=z{\bf 1}_{ i\in \llbracket 1, W\rrbracket}+\wt z\, {\bf 1}_{ i\notin \llbracket 1, W\rrbracket},
$$ 
and
\be\label{Def x} 
x_i\equiv (x_\zeta^{\b g})_i(z, \wt z):=(M_\zeta^{\b g})_i(z, \wt z)-m , \quad \mathbf M= \bx +m{\bf e}_1, \quad m:=m_{\rm sc}(\wt z+ {\rm i} 0^+\, ), \quad  {\bf e}_1=(1,1,\cdots, 1). 
\ee
Using the above notations and recalling Definition \ref{defHg}, we can rewrite \eqref{falvww} into the following form 
\be\label{111}
 (m+x_i)^{-1}= M_i^{-1}=-z_i-g_i-(S_0 M)_i
+\zeta({ \Sigma} M)_i = -z_i-g_i-(S_0 \bx )_i
 - m (S_0 {\bf e}_1)_i +\zeta({ \Sigma} \bx )_i+ \zeta m ({\Sigma} {\bf e}_1)_i.
\ee
Subtracting   $m^{-1}=-\wt z-m$ from the last equation and using $S_0 {\bf e}_1={\bf e}_1$, we get that
$$
m^{-1}-(m+x_i)^{-1}=g_i+(z_i-\wt z)+(S_0{\bf x})-\zeta  m  ({\Sigma}  {\bf e}_1 )_i -\zeta  ({\Sigma} \bx)_i.
$$
Then \eqref{111} is equivalent to 
\be\label{mGYY}
  [(1-m^2S_0) \bx]_i= m^2(g_i +(z_i-\wt z\, ))+m^2\left(\frac1{m+x_i}-\frac1m+\frac{x_i}{m^2}\right)-\zeta  m^3 ({\Sigma}  {\bf e}_1 )_i -\zeta m^2({ \Sigma} \bx)_i.
\ee
Define iteratively a sequence of vectors  $\bx^k\in \C^N$ such that $\bx^0=\b0\in \C^N$ and 
\be\label{niacin}
 \left[(1-m^2S_0) \bx^{k+1}\right]_i :=  m^2( g_i +(z_i-\wt z\, )) +m^2\left(\frac1{m+ (\bx^k)_i}-\frac1m+\frac{(\bx^{k})_i}{m^2}\right)-\zeta  m^3 ({ \Sigma}  {\bf e}_1)_i -\zeta  m^2({ \Sigma} \bx^{k})_i.
\ee
 In other words, \eqref{niacin} defines a mapping $h:l^\infty (\Z_N)\to l^\infty (\Z_N)$:
\be\label{iteration} 
\bx^{k+1}= h(\bx^k), \quad h_i(\bx):= \sum_j(1-m^2S_0)^{-1}_{ij} \left[ m^2( g_j +(z_j-\wt z\, )) + q(x_j) -\zeta  m^3 ({ \Sigma}  {\bf e}_1)_j -\zeta  m^2({ \Sigma} \bx^{k})_j\right],
\ee
where 
$$
q(x) := m^2\left(\frac1{m+ x } +\frac{x}{m^2} - \frac 1 m\right)  =\frac{x^2}{m+x}.
$$
Note by the assumptions of Lemma \ref{UE}, $c_\kappa \le m \le 1$ for some constant $c_\kappa>0$ depending only on $\kappa$. Then with \eqref{dust}, it is easy to see that there exists a sufficiently small constant $0< \alpha <c_\kappa/2$, such that $h$ is a self-mapping 
$$h: B_r \left(l^\infty(\Z_N)\right)\to B_r \left(l^\infty(\Z_N)\right), \quad B_r\left(l^\infty(\Z_N)\right):=\{\bx\in l^\infty(\Z_N): \|\bx\|_\infty \le r \},$$
as long as $r\le \alpha$ and 
\be\label{priori_cond}
\zeta+\| \b g\|_\infty+ |z-\wt z|\le c_r
\ee
for some constant $ c_r>0$ depending on $r$. Now it suffices to prove that $h$ restricted to $B_r \left(l^\infty(\Z_N)\right)$ is a contraction, which then implies that $\bx:=\lim_{k\to\infty} \bx^k$ exists and is a unique solution to \eqref{mGYY} subject to the condition $\|\b x\|_\infty \le r$.

From the iteration relation \eqref{iteration}, we obtain that
\be\label{k1k}
\bx^{k+1} - \bx^{k}= \frac {1} {1-m^2S_0} \left [ q (\bx^k) - q (\bx^{ k-1} ) \right]- \frac { \zeta  m^2} {1-m^2S_0} {\Sigma} ( \bx^{k}- \bx^{k-1}),
\ee
where $q(\bx)$ denotes a vector with components $q(x_i)$. 
Using $|q'(0)| = 0$ and \eqref{dust}, we get from \eqref{k1k} that
$$
\|\bx^{k+1}-\bx^k\|_\infty\le  C_\kappa \left(\zeta+\|\bx^{k }\|_\infty +\|\bx^{k-1 }\|_\infty  \right)\cdot \|\bx^{k }-\bx^{k-1}\|_\infty
$$
for some constant $C_\kappa>0$ depending only on $\kappa$. Thus we can first choose a sufficiently small constant $0<r<\al$ and then the constant $c_r>0$ such that $C_\kappa \left(c_r + 2r  \right)<1$, and $h$ is a self-mapping on $ B_r \left(l^\infty(\Z_N)\right)$ under the condition \eqref{priori_cond}. In other words, $h$ is indeed a contraction, which proves the existence and uniqueness of the solution.

Note that with \eqref{dust} and $\bx^0=\b0$, we get from \eqref{iteration} that
 $$ \|\bx^1\|_\infty=\OO \left(|z-\wt z|+\zeta+\|\b g\|_\infty\right).$$
With the contraction mapping, we have the bound 
$$\|\bx\|_\infty \le \sum_{k=0}^\infty \|\bx^{k+1}-\bx^k\|_\infty \le \frac{ \|\bx^1\|_\infty }{1-C_\kappa \left(\zeta + 2r  \right)}=\OO \left(|z-\wt z|+\zeta+\|\b g\|_\infty\right).$$
This gives the bound \eqref{bony}.

We now prove \eqref{dozy}. We have proved above that both $(M_\zeta^{\b g})_i(z, \wt z)$ and $(M_{\zeta'}^{\b g'})_i(z', \wt z\,')$ exist and satisfy \eqref{bony}. Denote by $m':=m_{\rm sc}(\wt z\,'+ {\rm i} 0^+\, )$ and $x'_i:=(M_{\zeta'}^{\b g'})_i(z', \wt z\,')-m' $. By \eqref{m'}, we have 
\be\label{m-m'}
|m'-m|={\rm O}(|\wt z - \wt z\,'|).
\ee 
Then using \eqref{mGYY} we can obtain that  
\begin{align*}
 \|  \bx'-\bx \|_\infty \le & C\|(1-m^2S_0)^{-1}\|_{L^\infty\to L^\infty}\cdot  \left\{ |\wt z-\wt z\,'|\cdot \left[ \|\bx'\|_\infty +  \|\b g'\|_\infty+ |z'-\wt z\,'| + \|\bx'\|_\infty^2 + \zeta' (1+\|\bx'\|_\infty)\right] \right.\\
& + \left. \left[  \|\b g-\b g'\|_\infty+ |z-z'| + |\wt z-\wt z\,'|+ |\zeta-\zeta'|(1+\|\bx'\|_\infty)+ \left(\zeta+\|\bx \|_\infty +\| \bx' \|_\infty  \right)\cdot \|\bx' - \bx \|_\infty\right]\right\} \\
\le & C\left(\zeta+\|\bx \|_\infty +\| \bx' \|_\infty  \right)\cdot \|\bx' - \bx \|_\infty + C \left(  \|\b g-\b g'\|_\infty+ |z-z'| + |\wt z-\wt z\,'|+ |\zeta-\zeta'|\right).
\end{align*}
Applying \eqref{bony} to both $(M_\zeta^{\b g})_i(z, \wt z)$ and $(M_{\zeta'}^{\b g'})_i(z', \wt z\,')$,
we see that for small enough $c$, 
$$ \|  \bx'-\bx \|_\infty \le C \left(  \|\b g-\b g'\|_\infty+ |z-z'| + |\wt z-\wt z\,'|+ |\zeta-\zeta'|\right).$$
Together with \eqref{m-m'}, we obtain \eqref{dozy} as desired. 
\end{proof}

\subsection{Proof of Lemma \ref{maxnorm}.}\
To prove Lemma \ref{maxnorm}, it suffices to prove the result for the case $ {\bf g}=0$, and we will describe how to relax to the condition ${\b g}={\rm O}(W^{-3/4})$ by using the Lipschitz continuity estimate \eqref{dozy} at the end of the proof. In preparation for the proof, we first prove the following lemma.

\begin{lemma}\label{lem63} 
Suppose that  $\b g=0$ and the  assumptions  \eqref{mouyzz}, \eqref{mouyzz2} and \eqref{jxw0} hold. 
Then there exist constants $c>0$ and $C>0$ such that 
\be\label{xiao}
 \left||(M_\zeta^{\b 0})_n|^2-|m|^2\right|\le C\left(|z-\wt z\, |+\zeta\right)e^{-c\frac{ |n|}W}, \quad n\in \Z_N, 
 \ee
  and
 \be\label{xiao1}
\frac1W \sum_{n \in \Z_N} (|m|^2|(M_\zeta^{\b 0})_n|^{-2} -1)\ge c( \im z - \im \wt z) - \zeta + \OO \left( N^{-\frac32\e_*} + N^{-\e^* }\im \wt z\right),
 \ee
where $m:=m_{\rm sc}(\wt z+\ii 0^+ )$. 
\end{lemma}

\begin{proof}[Proof of Lemma \ref{lem63}]  
 First with \eqref{gbzz} and the fact that $(S_0)_{ij}=0$ if $|i-j|\ge C_s W$,  we get that 
 $$
[ (1-m^2S_0)^{-1}]_{ij}-\delta_{ij} = [m^2 (1-m^2S_0)^{-1}S_0]_{ij}={\rm O}(W^{-1} ) \sum_{k\ge \frac{|i-j|}{C_sW}}   \left\| \frac{m^2S_0+\tau}{1+\tau}\right \|_{L^\infty\to L^\infty}^{ k} 
$$ 
Therefore with \eqref{gbzz2}, we obtain immediately that 
\be\label{M19}
 \left| [(1-m^2S_0)^{-1}]_{ij}-\delta_{ij}\right|\le CW^{-1} e^{-c\frac{|i-j|}{W}}
\ee
for some constants $c,C>0$.  As in the proof of Lemma \ref{UE}, with $\bx^k$ defined in \eqref{niacin}, we know that
\be\label{xn}
 x_n =   M_n-m   = x^1_n+\sum_{k\ge 1}(x^{k+1}_n- x^{k}_n), \quad M_n:=(M_\zeta^{\b 0})_{n}.
\ee
(Recall that we have proved that $x_n  = \lim_{k \to \infty} x^k_n$ in the proof of Lemma \ref{UE} above.) In particular, according to \eqref{niacin}, $\bx^1$ is given by 
\be\label{x1}
 [(1-m^2S_0) \bx^{ 1}]_i=  m^2(z_i-\wt z\, ) -\zeta  m^3 ({ \Sigma}  {\bf e}_1)_i  .
\ee
 Then with \eqref{M19} and \eqref{x1}, one can show that 
  \be\label{xn1}
    |x_n^1|\le C e^{-c\frac{|n|}{W}}\left(|z-\wt z\, |+\zeta\right), \quad n\in \Z_N. 
  \ee 
By \eqref{k1k} and \eqref{M19}, we have 
  $$
    |x_i^{k+1}-x_i^{k }|\le C \sum_j  \left(W^{-1} e^{-c\frac{|i-j|}{W}}+\delta_{ij}\right) \left[\left(|x^{k}_j|+|x^{k-1}_j|\right)|x^{k}_j-x^{k-1}_j| +  \zeta  \mathds{1}_{j\in \llbracket 1,W\rrbracket}\max_{j'\in \llbracket 1,W\rrbracket} |x^{k}_{j'}-x^{k-1}_{j'}| \right].
  $$
By  induction, it is easy to prove that there are constants  $c,C>0$ such that 
\be\label{zkkz}
    |x_n^{k+1}-x_n^{k }|\le C e^{-c\frac{|n|}{W}}\left(|z-\wt z\, |+\zeta\right)^{ k + 1} .
  \ee
Together with \eqref{xn1} and \eqref{xn}, this implies 
  \be\label{jzapp}
  |x_n| =   \left|  M_n -m \right|\le C\left(|z-\wt z\, |+\zeta\right)e^{-c\frac{ |n|}W}, \quad n\in \Z_N.
 \ee
This proves \eqref{xiao} since $\left||M_n|^2-|m|^2\right| \le |M_n^2 -m^2|$.

\medskip   
We now  prove \eqref{xiao1}. 
Using \eqref{xiao}, we have
\begin{equation} \label{insert1}
\begin{split}
& \frac1W\sum_{n \in \Z_N} (|m|^2|M_n|^{-2} -1) =\frac1{W |m|^2}\sum_{n \in \Z_N} \left(|m|^2 - |M_n|^{2} \right) + \OO \left(|z-\wt z|^2+\zeta^2 \right).
\end{split}
\end{equation}
\nc
By definition \eqref{Def x},  
$$
|M_n|^2=|m |^2+2\re (\bar m x_n) +|x_n|^2. 
$$ 
Then with \eqref{jzapp} we get that
\begin{align}\nonumber
 \frac1W \sum_n \left(|M_n|^2-|m|^2\right)&=   \frac{1}W\sum_n  \left[ 2\re (\bar m x_n)  +  |x_n|^2 \right]  =  \frac{2}W\sum_n  \re (\bar m x_n)  + {\rm O}( |z-\wt z|^2+\zeta^2 ). \nonumber
\end{align}
By \eqref{mouyzz} and  \eqref{mouyzz2}, we have
$$ \zeta^2+|\re(z-\wt z)|^2\le T^2 + r^2 \le N^{-3\e_*/2},\quad 0\le \im\wt z\le \im z\le N^{- \e^*},$$
which implies that  
$$
\zeta^2 + |z-\wt z\, |^2  \le  \zeta^2+|\re(z-\wt z)|^2  +  \im (z-\wt z )^2  \le N^{-3\e_*/2} + N^{-\e^*} \im (z-\wt z).
$$
Then using \eqref{xn} and \eqref{zkkz}, we obtain that 
\begin{align}\nonumber
\frac1{W}\sum_{n} \left( |M_n|^{2} - |m|^2 \right) &= \frac{2}W\sum_n \re (\bar m x_n) 
  +\OO \left( N^{-\frac32\e_*} + N^{-\e^*} \im (z-\wt z)\right)\\\label{waccp}
 &= \frac{2}W\sum_n \re \left(\bar m x^1_n \right) +\OO \left( N^{-\frac32\e_*} + N^{-\e^*} \im (z-\wt z)\right).
\end{align}
 Summing \eqref{x1} over $i$, we get that (recall that we take $\b g=0$)
$$
 (1-m^2) \sum_i x^{1}_i :=  m^2  \sum_i (z_i-\wt z\, )
 -\zeta  m^3 \left(W+ 1\right)  = m^2 W (z - \wt z) -\zeta  m^3 W+ {\rm O}(1),
$$
where we used that $\sum_i (S_0)_{ij}=1$ and $( \Sigma {\bf e}_1)_i = 1+ W^{-1}$ for $i\in\llbracket 1,W\rrbracket$. Thus for the second term in the second line of \eqref{waccp}, we have
 \begin{align}\label{631}
  \sum_n \re (\bar m x^1_n)&
= |m|^2W\re \left(\frac{(z-\wt z\, )m-\zeta m^2}{1-m^2}\right) +   {\rm O}(1) 
\nonumber \\
&=|m^2|W\left(\frac{\zeta}{2}-\frac{\im z-\im \wt z\, }{\sqrt{4-|\re \wt z|^{\,2}}}+\OO \left( N^{-\e^*} \im \wt z \,\right)\right) +   {\rm O}(1),
\end{align}
where we have used the following special properties of $m(\wt z+\ii 0^+)$ when $\wt z$ is a real number, in which case $m(\wt z+\ii 0^+)$ has unit modulus:
\be
\re \frac{m(a^+)} {1-m^2 (a^+)} = 0, \quad  \im \frac{m(a^+)}{1-m^2(a^+)}=\frac{1}{  \sqrt{4-a^2}},\quad \re  \frac{m^2(a^+)}{1-m^2(a^+)}=-\frac12,\quad |a|<2, \ \ a^+:=a+\ii 0^+.
\ee
Here the error $\OO \left( N^{-\e^*} \im \wt z \,\right)$ in \eqref{631} is due to $|m(\wt z)- m(\re \wt z + {\rm i} 0^+)| \le C \im \wt z $.
Inserting \eqref{631}  into \eqref{waccp}, we obtain that for some constant $c>0$,
$$ \frac1{W}\sum_{n} \left( |M_n|^{2} - |m|^2 \right)  \le -c (\im z-\im \wt z) + \zeta |m|^2 + \OO \left( N^{-\frac32\e_*} +  N^{- \e^*}\im \wt z \right),$$
which, together with \eqref{insert1}, proves  \eqref{xiao1}.
 \end{proof}

With Lemma \ref{lem63}, we now finish the proof of Lemma \ref{maxnorm}.

\begin{proof}[Proof of Lemma \ref{maxnorm}]\nc 
We first assume that $\b g=0$. With \eqref{gbzz2} and a perturbation argument, we can show that
$$\left \|\left (\frac{M^2S+\tau}{1+\tau}\right)^2\right \|_{L^\infty\to L^\infty}<  1-c$$
for some constant $c>0$. Then \eqref{tianYz} can be proved as in \eqref{M19}. 
Our main task is to prove \eqref{zlende}. Assume that  
\be\label{afafs}
(1-|M|^2S)  \bu^0=\bv^0
\ee
for some vectors $\bu^0,\bv^0\in \bR^N$. Multiplying \eqref{afafs}  with $\bu^0  |M|^{-2 }$ from the left and using the definition of $S$, we obtain that
\be\label{nhx}
\sum_i (|M_i|^{-2} -1)|\bu^0_i|^2+\sum_{1\le i\le W}  \zeta (1+W^{-1})|\bu^0_i|^2  + \frac12\sum_{i,j} S_{ij}\left(\bu^0_i-\bu^0_j\right)^2
= (\bu^0, |M|^{-2 }\bv^0).
\ee
We define a symmetric operator $H: L^2(\mathbb T)\mapsto L^2(\mathbb T)$, where $\mathbb T:=\llbracket-(\log N)^4 W, (\log N)^4 W \rrbracket$ and
$$H := H_{ 0}+H_{ 1},$$
with
$$H_0:  \ \ (\bu, H_0 \bv)=\frac14\sum_{  i,j \in \mathbb T} S_{ij}\left(\bu_i-\bu_j\right)\left(\bv_i-\bv_j\right),\quad \bu,\bv\in L^2(\mathbb T),$$ 
and
$$
H_1:  \ \ (H_1)_{ij}:= \delta_{ij}\left[(|M_i|^{-2} -1)+\zeta \, {\bf1}_{1\le i\le W}  (1+W^{-1})\right].$$
For any vector $\bu$, we denote by $\bu|_{\mathbb T}$ the restriction of $\bu$ to $L^2(\mathbb T)$. Then with \eqref{xiao} and the fact that $|m|\le 1$, we can rewrite \eqref{nhx} as
\be\label{key step}
(\bu^0|_{\mathbb T}, H\bu^0|_{\mathbb T})+ \frac14\sum_{i,j} S_{ij}\left(\bu^0_i-\bu^0_j\right)^2
\le (\bu^0, |M|^{-2 }\bv^0)+{\rm O}(N^{-10}) \|\bu^0\|^2_2 . 
\ee
 
First we claim that 
\be\label{adfa111}
H\ge c\im z(\log N)^{-4}
\ee
for some constant $c>0$. With Temple's inequality, we have the following estimate on the ground state energy of $H$:
\be\label{temple}
H\ge E_0(H)\;\ge\; \langle H \rangle_\phi-\frac{ \langle(H )^2\rangle_\phi-\langle H \rangle_\phi^2}{E_1(H )-\langle H \rangle_\phi},
\ee
for any $\phi\in L^2(\mathbb T)$ such that $ \|\phi\|_2=1$ and $\langle H \rangle_\phi<E_1(H)$, where $E_0(H)$ and $E_1(H)$ are the lowest two eigenvalues of $H$. Applying min-max principle to $H\ge H_0-\|H_{1 }\| _{L^2\to L^2} $, we obtain that
\be\label{minmax}
E_1(H) \;\ge\; E_1(H_{0 })- \|H_{1 }\| _{L^2\to L^2}.
\ee
By \eqref{xiao}, we have $\|H_{1 }\| _{L^2\to L^2}=\OO \left(|z-\wt z\, |+\zeta+\im \wt z\right)$. We then claim that
\be\label{spectral gap}
E_1(H_{0 }) \ge c(\log N)^{-13}
\ee
for some constant $c>0$. Recall that $ S\equiv S^{\zeta}= S_0 -\zeta \Sigma$ with 
$$\frac14\sum_{  i,j \in \mathbb T} \Sigma_{ij}\left(\bu_i-\bu_j\right)^2 \le 1,\quad \forall \bu \in L^2(\mathbb T), \ \ \|\bu\|_2=1.$$ 
Then again by min-max principle, it suffices to prove the following lemma.
\begin{lemma}\label{small lem}
For $s_{ij}$ satisfying \eqref{bandcw0}-\eqref{bandcw1}, there exists a constant $c>0$ such that
$$
\frac14\sum_{  i,j \in \mathbb T} s_{ij}\left(\bu_i-\bu_j\right)^2 \ge c(\log N)^{-13},\quad \forall \bu \in L^2(\mathbb T), \ \ \|\mathbf u\|_2 =1, \ \ \mathbf u \perp (1,1,\cdots,1).
$$
\end{lemma}
We postpone its proof until we finish the proof of Lemma \ref{maxnorm}. We now choose the trial state $\phi$ as a constant vector in \eqref{temple}, i.e., 
 $$\phi_0=\frac{1}{\sqrt{|\mathbb T|}}(1,1,\cdots,1).$$  
Then by definition, $H_{0}\phi_0={\bf 0} $ and $\langle H \rangle_{\phi_0} \le \|H_{1 }\| _{L^2\to L^2} \ll E_1(H)$ by \eqref{minmax} and \eqref{spectral gap}. Then by \eqref{temple} and \eqref{minmax}, we have
\be\label{temple2}
H\;\ge\; \langle H_{1}\rangle_{\phi_0}-\frac{\|H_{1 }\|^2_{L^2\to L^2}}{E_1(H_{ })-\langle H_{1}\rangle_{\phi_0}}  \;\ge\;  
\langle H_{1}\rangle_{\phi_0}-\frac{\|H_{1 }\|^2_{L^2\to L^2}}{E_1(H_{0 })-2\|H_{1 }\| _{L^2\to L^2} }. 
 \ee
By the definition of $H_1$, we have
\begin{align*}
\langle H_{1}\rangle_{\phi_0}&=\frac{1}{|\mathbb T|}\sum_{n\in \mathbb T}\left[(|M_n|^{-2} -1)+\zeta \, {\bf1}_{n\in \llbracket 1, W\rrbracket}  (1+W^{-1})\right]\\
&= \frac{1}{|\mathbb T|}\sum_{n\in \mathbb T}(1-|m|^2)|M_n|^{-2} +\frac{1}{|\mathbb T|}\sum_{n\in \mathbb T}(|m|^2|M_n|^{-2} -1)+ \frac{\zeta(W+1)}{|\mathbb T|} \\
& \ge c\im \wt z + {\rm O}(N^{-10}) +\frac{1}{|\mathbb T|}\sum_{n\in \mathbb Z_N}(|m|^2|M_n|^{-2} -1)+\frac{\zeta(W+1)}{|\mathbb T|} \\
& \ge c\im z(\log N)^{-4} + \OO \left( N^{-\frac32\e_*} + N^{-\e^*}\im z\right),
\end{align*}
where we used \eqref{xiao} and $|m|^2 \le 1 - c\im \wt z$ (by \eqref{m is bounded}) in the third step, and 
\eqref{xiao1} in the last step.
Together with \eqref{temple2}, $\|H_{1 }\|^2_{L^2\to L^2}={\rm O}( N^{-\frac32\e_*} + N^{-\e^*}\im z)$ and \eqref{spectral gap}, this proves \eqref{adfa111}.

 With \eqref{adfa111}, \eqref{key step} gives that for some $c>0$, 
$$
c\im z(\log N)^{-4} \sum_{i\in \mathbb T}|\bu^0_i|^2 +\frac14\sum_{i,j} S_{ij}\left(\bu^0_i-\bu^0_j\right)^2
\le (\bu^0, |M|^{-2 }\bv^0)+{\rm O}(N^{-10}) \|\bu^0\|^2_2 .
$$
Now for some fixed $ i_0\in \Z_N$, we choose $\bv^0= S{\bf e}_{i_0}$. Then the above inequality becomes 
\be\label{yjzzex}
c\im z(\log N)^{-4} \sum_{i\in \mathbb T}|\bu^0_i|^2 +\frac14\sum_{i,j} S_{ij}\left(\bu^0_i-\bu^0_j\right)^2
\le (S|M|^{-2 }\bu^0)_{i_0} +{\rm O}(N^{-10}) \|\bu^0\|^2_2 .
\ee
In the following, we suppose $\|\bu^0\|_\infty \gg W^{-1}$, otherwise the proof is done. Since for any $i\in \Z_N$,
\be\label{u0eqn}
(\bu^0- |M|^2S   \bu^0)_i = (S{\bf e}_{i_0})_i ={\rm O}(W^{-1}), 
\ee
we must have 
$$\|\bu^0\|_\infty \asymp \|S\bu^0\|_\infty.$$  
Now we decompose $\bu^0$ as follows:  
$$
\bu^0_i =u+ \wt \bu_i, \quad \text{with } u=\frac1N\sum_{i\in \Z_N}\bu^0_i , \ \ \sum_i\wt \bu_i=0 .
$$
Suppose  $|u|\ge 10 \| \wt \bu\|_\infty$, then we have 
$$
\max_i |\bu^0_i| \le 2\min _i |\bu^0_i|  .
$$
Together with \eqref{yjzzex}, it implies that if $| u|\ge 10 \| \wt \bu\|_\infty$, then 
\be\label{part1}
\|\bu^0\|_\infty\le  2|u| \le C(W{  \im z})^{-1}.
\ee
On the other hand, if $| u|\le 10 \| \wt \bu\|_\infty$, with \eqref{afafs}, \eqref{xiao} and the definition of $S$ in Definition \ref{defHg}, we get that
\be \label{u0eqn2}
\wt \bu - |M|^2S   \wt \bu=\OO \left(W^{-1}+(\zeta+|z-\wt z|) |u|\right).
\ee
Then in this case, with \eqref{u0eqn} and \eqref{u0eqn2} it is easy to see that
\be\label{asdfjyyy}
 \|\bu^0\|_\infty \asymp \|S\bu^0\|_\infty \asymp  \|\wt \bu \|_\infty \asymp \|S\wt \bu\|_\infty \asymp  \|S_0\wt \bu\|_\infty.
\ee
By \eqref{sumsone}, we have 
$$
\sum_j (S_0\wt \bu\,)_j=0,
$$
which implies
\be\label{max diff}
\|S_0  \wt \bu \|_\infty \le \max_{i,j}\left|(S_0  \wt \bu \,)_j -(S_0  \wt \bu \,)_i\right| .
\ee
 Using (\ref{sumsone}), for fixed $i\le j\in \Z_N$ we have
\be\label{scwzjl}
\left|(S_0  \wt \bu \,)_j -(S_0  \wt \bu \,)_i\right| ^2 = \left|\sum_{x,y} (S_0)_{ix}(S_0)_{jy}(\wt \bu _{x}-  \wt \bu _{y})\right| ^2\le \sum_{x,y}(S_0)_{ix}(S_0)_{jy} |  \wt \bu _{x}-  \wt \bu _{y}|^2.
\ee

The lower bound in \eqref{bandcw1} shows that $S_0$ has a core, i.e., there is a constant $c_s>0$ such that 
$(S_0)_{xy}\ge c_sW^{-1}$ if $|x-y|\le W$. 
Then for any fixed $i\le j\in \Z_N$,
we choose $x_0, x_1, x_2, \cdots ,x_n$ for some $n={\rm O}(N/W)$ such that 
$$
i=x_0\le x_1\le x_2 \le \cdots \le x_{n-1}\le x_n=j, \quad \text{ with }\ \    W/3\le |x_k-x_{k+1}|\le W/2, \ \forall k. 
$$
Furthermore, set $x'_0=x$ and $x'_n=y$. Clearly for any choices of $x'_k$, $1\le k\le n-1$, we have
$$
    \wt \bu  _y -    \wt \bu  _x=\sum_{k=1}^n  \left( \wt \bu  _{x'_k} -    \wt \bu  _{x'_{k-1}}\right) \Rightarrow | \wt \bu  _y -    \wt \bu  _x|^2\le \frac {CN}W \sum_{k=1}^n \left|    \wt \bu  _{x'_k} -    \wt \bu  _{x'_{k-1}}\right|^2 .
$$
For our {\rm goal}, we will choose  $x'_k$'s such that
 $$
 x'_k \in \llbracket x_k-W/4, x_k+W/4\rrbracket, \quad 1\le k \le n-1. 
 $$  
Taking averaging over all $x'_k$, $1\le k \le n-1$, in the above regions,  we get that 
 $$
|    \wt \bu  _y -    \wt \bu  _x|^2\le \frac NW \left(\text{Average}_{x'_1, x'_2, \cdots, x'_{n-1}}\right) \sum_{k=1}^n \left|    \wt \bu  _{x'_k} -    \wt \bu  _{x'_{k-1}}\right|^2.
$$
Note that by our choices, we always have $ |x'_k-x'_{k-1}|\le W$ and $S_{x'_k x'_{k-1}}\ge \frac12 c_s W^{-1}$ for $2\le k \le n-1$, which gives that
\begin{align*}
\text{Average}_{x'_{k-1}, x'_k} \left| \wt \bu  _{x'_k} -    \wt \bu  _{x'_{k-1}}\right|^2 & \le \frac{4}{W^2} \sum_{x_k' ,x_{k-1}'\in  \llbracket x_{k-1} -W/4, x_k+W/4\rrbracket} \left|    \wt \bu  _{x'_k} -    \wt \bu  _{x'_{k-1}}\right|^2 \\
& \le \frac{8c_s^{-1}}{W} \sum_{x_k' ,x_{k-1}'\in  \llbracket x_{k-1} -W/4, x_k+W/4\rrbracket}S_{x'_k x'_{k-1}}\left|    \wt \bu  _{x'_k} -    \wt \bu  _{x'_{k-1}}\right|^2.
\end{align*}
Together with \eqref{scwzjl}, we get that for some constant $C>0$,
\begin{align*}
\left|(S_0  \wt \bu \,)_j -(S_0  \wt \bu \,)_i\right| ^2 & \le \sum_{x,y}(S_0)_{ix}(S_0)_{jy} \left[\frac{CN}{W^2}\sum_{k=2}^{n-1} \sum_{x_k',x_{k-1}' \in \llbracket x_{k-1}-W/4, x_k+W/4\rrbracket}S_{x'_k x'_{k-1}}\left|    \wt \bu  _{x'_k} -    \wt \bu  _{x'_{k-1}}\right|^2 \right]\\
& + \sum_{x,y}(S_0)_{ix}(S_0)_{jy}  \frac{CN}{W}\left[\frac{2}{W} \sum_{x':|x'-x_1| \le W/4} \left|    \wt \bu  _{x'} -    \wt \bu  _{x}\right|^2 + \frac{2}{W} \sum_{y':|y'-x_{n-1}| \le W/4}\left|    \wt \bu  _{y} -    \wt \bu  _{y'}\right|^2 \right] .
\end{align*}
For the first term on the right-hand side, we have
$$\sum_{k=2}^{n-1} \sum_{x_k',x_{k-1}' \in \llbracket x_{k-1}-W/4, x_k+W/4\rrbracket}S_{x'_k x'_{k-1}}\left|    \wt \bu  _{x'_k} -    \wt \bu  _{x'_{k-1}}\right|^2 \le C\sum_{ k,l\in \Z_N}S_{kl}\left(\wt \bu _k-\wt \bu _l\right)^2 .$$
For the terms in the second line, we notice that
$$|x'-x| \le |x'-x_1| + |x_1-i| + |i-x| \le C_sW+W$$
for all $x'$ such that $|x'-x_1| \le W/4$, where $C_s$ is the constant appeared in \eqref{bandcw1}. Then we can subdivide the interval $\llbracket x,x' \rrbracket$ or $\llbracket x',x\rrbracket$ into subintervals with lengths $\le W/2$, and proceed as above to get
$$ \sum_x (S_0)_{ix}\frac{2}{W} \sum_{|x'-x_1| \le W/4}\left|    \wt \bu  _{x'} -    \wt \bu  _{x}\right|^2 \le \frac{C}{W}\sum_{1\le k,l\le N}S_{kl}\left(\wt \bu _k-\wt \bu _l\right)^2  $$
for some constant $C>0$ that is independent of the choice of $x'$. In sum, we have obtained that
$$
\left|(S_0  \wt \bu \,)_j -(S_0  \wt \bu \,)_i\right| ^2
\le \frac{CN}{W^2}  \sum_{1\le k,l\le N}S_{kl}\left(\wt \bu _k-\wt \bu _l\right)^2 
= \frac{CN}{W^2}  \sum_{1\le k,l\le N}S_{kl}\left( \bu _k^0- \bu _l^0\right)^2 .$$
Then from \eqref{asdfjyyy} and \eqref{max diff}, we obtain that
$$\|\bu^0\|_\infty^2 \le \frac{CN}{W^2}  \sum_{1\le k,l\le N}S_{kl}\left( \bu _k^0- \bu _l^0\right)^2.$$
Plugging it into \eqref{yjzzex}, we get that if  $| u|\le 10 \| \wt \bu\|_\infty$, then
\be\label{part2} 
\frac{W^2}{N}\|\bu^0\|_\infty^2 \le C \sum_{1\le k,l\le N}S_{kl}\left( \bu _k^0- \bu _l^0\right)^2 \le C\|\bu^0\|_\infty  +{\rm O}(N^{-10}) \|\bu^0\|^2_2\Rightarrow \|\bu^0\|_\infty\le \frac{C N}{W^{2}}. \ee
In sum, by our choice of $\bv^0= S{\bf e}_{i_0}$ and \eqref{afafs}, we obtain from \eqref{part1} and \eqref{part2} that  
$$
\left\|\left(1-S|M|^2\right)^{-1}S\right\|_{\max} \le C\left( \frac{1}{W\im z}+\frac{N}{W^2}\right),
$$
which completes the proof of \eqref{zlende} in the case with $\b g =0$.

Given any $\b g\in \mathbb R^N$ such that $\|\b g\|_\infty \le W^{-3/4}$, we can write 
$$M_\zeta^{\b g} = M_\zeta^{\b 0} + \mathcal E,$$
where $\mathcal E$ is a diagonal matrix with $\max_i |\mathcal E_{ii}| = {\rm O}(\|\b g\|_\infty) ={\rm O}( W^{-3/4})$ by the Lipschitz continuity estimate  \eqref{dozy}. Then \eqref{tianYz} can be obtained by combing \eqref{tianYz} in the case $\b g=0$ with a standard perturbation argument. For \eqref{zlende}, we write
\be\label{perturb}
\left(1-S|M_\zeta^{\b g}|^2\right)^{-1}S = \left(1-S|M_\zeta^{\b 0}|^2\right)^{-1}S + \left(1-S|M_\zeta^{\b 0}|^2\right)^{-1}S (|M_\zeta^{\b g}|^2-|M_\zeta^{\b 0}|^2)\left(1-S|M_\zeta^{\b g}|^2\right)^{-1}S. 
\ee
Using \eqref{zlende} in the case $\b g=0$ and the bound
$$\left\| \left(1-S|M_\zeta^{\b 0}|^2\right)^{-1}S\right\|_{L^\infty \to L^\infty} \le N\left\| \left(1-S|M_\zeta^{\b 0}|^2\right)^{-1}S\right\|_{\max}, $$
we get from \eqref{perturb} that
$$\left\|\left(1-S|M_\zeta^{\b g}|^2\right)^{-1}S\right\|_{\max} \le \left\|\left(1-S|M_\zeta^{\b 0}|^2\right)^{-1}S\right\|_{\max} +  {\rm O} \left( \left( \frac{N}{W\im z}+\frac{N^2}{W^2}\right)W^{-3/4}\right).\left\|\left(1-S|M_\zeta^{\b g}|^2\right)^{-1}S\right\|_{\max}.$$
Together with \eqref{jxw0}, this implies \eqref{zlende} for any $\b g$ such that $\|\b g\|_\infty \le W^{-3/4}$.
\end{proof}
 
\begin{proof}[Proof of Lemma \ref{small lem}]
Since the matrix $S_0=(s_{ij})$ has a core by \eqref{bandcw1}, it suffices to prove that 
\be\label{core}
\sum_{  i ,j \in \mathbb T} \hat s_{ij}\left(\bu_i-\bu_j\right)^2 \ge c(\log N)^{-13},\quad \forall \bu \in L^2(\mathbb T), \ \ \|\mathbf u\|_2 =1, \ \ \mathbf u \perp (1,1,\cdots,1),
\ee
where 
$$\hat s_{ij}:=\frac1W\mathds{1}_{|i-j|\le W}.$$
Then we define the following two symmetric operators $F_{0,1}: L^2(\mathbb T)\mapsto L^2(\mathbb T)$ such that for any $\bu,\bv\in L^2(\mathbb T)$,
\begin{align*}
(\bu, F_0 \bv)=\frac1{W(\log N)^5}\sum_{  i,j \in \mathbb T, |i-j|_{\mathbb T} \le W} \left(\bu_i-\bu_j\right)\left(\bv_i-\bv_j\right), 
\end{align*}
where $|\cdot |_{\mathbb T}$ denotes the periodic distance on $\mathbb T$, and
\begin{align*}
(\bu, F_1 \bv)= \sum_{i,j \in \mathbb T} \tilde s_{ij}\left(\bu_i-\bu_j\right)\left(\bv_i-\bv_j\right) , \quad \tilde s_{ij}:= \hat s_{ij}- \frac{1}{W(\log N)^5} \mathds{1}_{|i-j|_{\mathbb T}\le W}.
\end{align*}

We first show that for some constant $c>0$,
\be\label{F0}
E_1(F_0) \ge c(\log N)^{-13},
\ee
 where $E_1(F_0)$ denotes the second lowest eigenvalue of $F_0$. Without loss of generality, we can regard $F_0$ as an operator on $L^2(\mathbb T,\mathbb C)$ consisting of {\it complex} $L^2$ vectors. Since $F_0$ is a periodic operator on $L^2(\mathbb T,\mathbb C)$, its eigenvectors are the unit complex vectors with Fourier components:
$$\mathbf w_p:\quad (\mathbf w_p)_k := \frac1{\sqrt{|\mathbb T|}}e^{\ii p k},  \quad k\in \mathbb T, \quad \text{with}\quad p=\frac{2\pi n}{|\mathbb T|}, \ \ n\in \mathbb T .$$
Then for any $p\ne 0$, we have 
\begin{align*}
(\mathbf w_p, F_0 \mathbf w_p)&=\frac{1}{W(\log N)^5}\sum_{  |k-l|_{\mathbb T} \le W} \left|(\mathbf w_p)_k-(\mathbf w_p)_l\right|^2  = \frac{1}{|\mathbb T|W(\log N)^5} \sum_{  |k-l|_{\mathbb T} \le W}\left[2-2\cos(p(k-l))\right]  \\
&=\frac{1}{W(\log N)^5} \sum_{  |n| \le W}\left[2-2\cos(pn)\right] \ge \frac{c}{W(\log N)^5}\frac{W^3}{|\mathbb T|^2} \ge c(\log N)^{-13}. 
\end{align*}
This proves \eqref{F0}.

We now show that $F_1$ defines a positive operator. For simplicity of notations, we let $L=|\mathbb T|$ and shift  $\mathbb T$ to $\mathbb T:=\llbracket 1,L\rrbracket$. Then $\tilde s_{ij}$ can be written as
\begin{align}\label{tilde sij}
\quad \tilde s_{ij}=\left(1-(\log N)^{-5}\right)\hat s_{ij}  - \frac1{W(\log N)^5}\left(\mathds{1}_{ 1\le i \le W, L-W+i \le j \le L} + \mathds{1}_{ 1\le j \le W, L-W+j \le i \le L}\right) .
\end{align}
Fix any $\mathbf u\in L^2(\mathbb T)$. The following proof is very similar to the one below \eqref{scwzjl}, so we shall omit some details. For any fixed $1\le i\le W$ and $L-W\le j \le L$, we choose $x_0, x_1, \cdots, x_n$ for some $n={\rm O}((\log N)^4)$ such that 
$$
i=x_0\le x_1\le x_2 \le \cdots \le x_{n-1}\le x_n=j, \quad \text{ with }\ \    W/3\le |x_k-x_{k+1}|\le W/2, \ \forall k. 
$$
Moreover, we set $x'_0=i$ and $x'_n=j$. Then we can get as before that  
 $$
| \bu  _i -    \bu _j |^2\le C(\log N)^4 \left(\text{Average}_{x'_1, x'_2, \cdots, x'_{n-1}}\right) \sum_{k=1}^n \left|    \bu  _{x'_k} -  \bu  _{x'_{k-1}}\right|^2,
$$
where we took average over all $x'_k\in \llbracket x_k-W/4, x_k+W/4\rrbracket$, $1\le k \le n-1$.
Note that by our choices, we always have $ |x'_k-x'_{k-1}|\le W$ and $\hat s_{x'_k x'_{k-1}}= W^{-1}$ for $1\le k \le n$, which gives that
\begin{align*}
&\, \frac1{W(\log N)^5}\sum_{1\le i \le W, L-W\le j \le L}\left|\bu_i - \bu_j\right| ^2  \\
\le &\, \frac1{W(\log N)^5}\sum_{1\le i \le W, L-W\le j \le L} \left[\frac{C(\log N)^4}{W}\sum_{k=2}^{n-1} \sum_{x_k',x_{k-1}' \in \llbracket x_{k-1}-W/4, x_k+W/4\rrbracket}\hat s_{x'_k x'_{k-1}}\left| \bu  _{x'_k} - \bu  _{x'_{k-1}}\right|^2 \right]\\
+ &\, \frac1{W(\log N)^5} \sum_{1\le i \le W, L-W\le j \le L} C(\log N)^4\left[ \sum_{x:|x-x_1| \le W/4} \hat s_{x i}\left|    \bu  _{x} -  \bu  _{i}\right|^2 + \sum_{y:|y-x_{n-1}| \le W/4} \hat s_{jy}\left|  \bu  _{y} -  \bu  _{j}\right|^2 \right] \\
\le &\, C(\log N)^{-1} \sum_{ k,l\in \mathbb T}\hat s_{kl}\left(\bu _k-\bu _l\right)^2 .
\end{align*}
Then by \eqref{tilde sij}, it is easy to see that $F_1$ is a positive operator. Thus by min-max principle we have
 $$E_1(F_0 + F_1)\ge E_1(F_0),$$
 which proves \eqref{core} together with \eqref{F0}.
\end{proof}

  \end{document}